\documentclass[11pt]{amsart}
\usepackage{geometry}               
\usepackage[parfill]{parskip}                
\usepackage{graphicx}
\usepackage{amssymb}
\usepackage{epstopdf}

\newtheorem{theorem}{Theorem}
\newtheorem{prop}[theorem]{Proposition}
\newtheorem{cor}[theorem]{Corollary}

\theoremstyle{plain}
\newtheorem{definition}[theorem]{Definition}
\newtheorem{example}[theorem]{Example}
\newtheorem{remark}[theorem]{Remark}

\title{Properties of the poset of Dyck paths ordered by inclusion}
\author{Jennifer Woodcock }
\begin{document}
\begin{titlepage}
\text{ }
\\
\\
\\
\\
\\
\\
\begin{flushright}
\bfseries
PROPERTIES OF THE POSET OF DYCK PATHS\newline
ORDERED BY INCLUSION\newline
\text{ }
\\
\end{flushright}
\text{  }
\\
\\
\\
\\
\\
\\
\\
\\
\\
\\
\\
\\
\\
\\
\begin{flushright}
by: Jennifer Woodcock\newline
for: Dr. Mike Zabrocki\newline
Math 6002\newline
York University\newline
\text{ }\newline
\text{ }\newline
\text{ }\newline
\text{ }\newline
\text{ }\newline
\begin{center}
August 2008
\end{center}
\end{flushright}
\end{titlepage}


\newpage
\thispagestyle{empty}
\begin{figure}[t]
	\centering
		\includegraphics[width=6in]{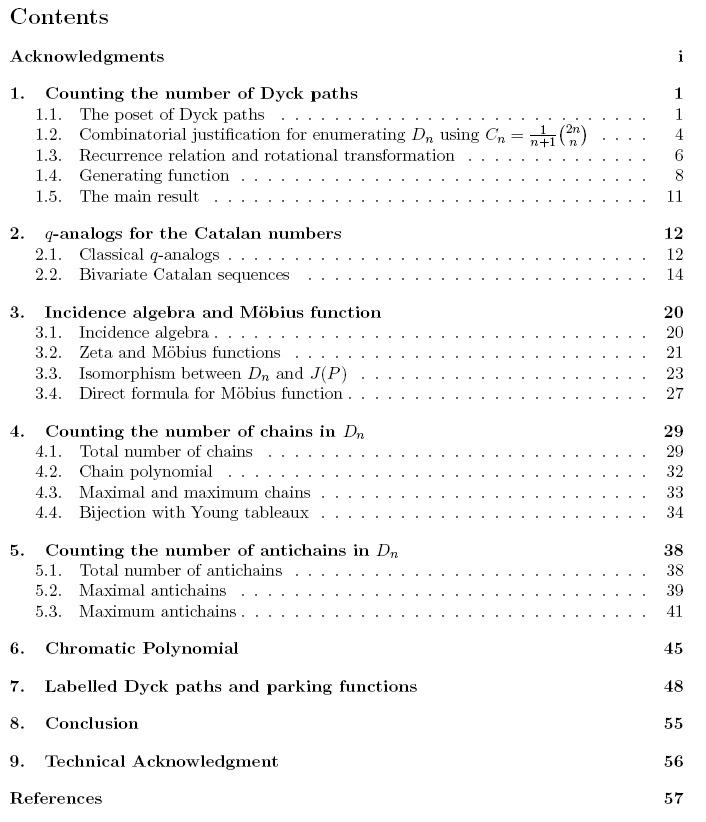}
\end{figure}
\text{  }
\\
\\
\newpage

\thispagestyle{plain}
\pagenumbering{roman}
\section*{Acknowledgments}
\text{  }
\\
\\

As this is perhaps the closest I will come to writing a thesis, I'm taking the liberty of acknowledging the people who have made an intimidating and ambitious project become such a rewarding experience.\newline
\\
First and foremost, I am infinitely grateful to Professor Mike Zabrocki, who believed in me, challenged me, and had the patience of a saint when it came to answering my many, many questions. But aside from helping with the computer technical work and the finicky details of mathematical writing, Mike's influence ran deeper. Like any excellent teacher, he knows it's about more than just giving students the knowledge and the skills. He really helped me to develop confidence in my abilities -- an invaluable and most treasured gift. Thanks, Mike.\newline
\\
Additionally, I couldn't have done this without the support of my family. To Bob, Ryan and Sarah, my deepest thanks for putting up with an absentee wife and mom on so many occasions. No more doodling Dyck paths on restaurant menus when we're out, I promise!\newline
\\
A special thank you goes out to Nichole, who took great care of Ryan and Sarah. Your babysitting services meant more than you know.\newline
\\
Sincere thanks go to Professors Hosh Pesotan and Jack Weiner of the University of Guelph's Department of Mathematics  and Statistics for their proofreading and advice on structuring a math paper. I enjoyed learning along with you during our discussions about all things Catalan and Dyck.\newline
\\
Finally, I also wish to thank: my parents for driving me to succeed; fellow M.A. students and friends for comic relief, carpooling, and shoulders to cry on; programmers at Maplesoft and MuPAD: Combinat for creating software that makes research much less tedious; math teachers and profs. both past and present for their inspiration; and Professor Augustine Wong, Graduate Programme Director of Mathematics and Statistics at York, for his help in getting me started on this journey.\newline
\text{ }
\\
\\
\\
Jennifer Woodcock\newline
August 2008

\newpage

\pagenumbering{arabic}

%
\maketitle
\section{Counting the number of Dyck paths}

\subsection{The poset of Dyck paths}
\text{  }
\\
\\
The Catalan numbers, $C_n=\frac{1}{n+1}\binom{2n}{n}$, are a sequence of integers named after the nineteenth century mathematician Eugene Catalan \cite{MB}.  
They are included as A000108 in \cite{OLEIS}: 1, 1, 2, 5, 14, 42, 132, 429, 1430, 4862, 16796, 58786, 208012, 742900, 2674440, 9694845, 35357670, 129644790, 477638700, 1767263190, 6564120420, 24466267020, 91482563640, 343059613650, 1289904147324.\newline
Objects enumerated by the Catalan numbers occur ubiquitously in the field of combinatorics. In fact, Stanley in \cite{RS2} and \cite{webStanley} has compiled a list of over 165 of these. They include many diverse (and at first glance, unrelated) sets including  graphical objects like the number of rooted, unlabelled binary trees on $n$ vertices, sequences of integers such as $1\leq x_1 \leq \ldots \leq x_n$ such that $x_i \leq i$, geometric items such as the number of ways of splitting an $(n+2)$-gon into $n$ triangles using $n-1$ non-crossing diagonals, as well as several types of lattice paths subject to a variety of conditions.  It is in this latter category that we find the Dyck paths.
\text{ }
\\
\begin{definition}
A Dyck path is a lattice path in the $n \times n$ \itshape square consisting of only north and east steps and such that the path doesn't pass below the line $y=x$ (or main diagonal) in the grid.  It starts at $(0,0)$ and ends at $(n,n)$.  A walk of length $n$ along a Dyck path consists of $2n$ steps, with $n$ in the north direction and $n$ in the east direction.  By necessity the first step must be north and the final step must be east.\normalfont
\end{definition}

\noindent Dyck paths are named after the German mathematician, Walther von Dyck (1856-1934) for whom the Dyck language of formal language theory is also named \cite{webDL}. As described in \cite{NL2}, there is a clear bijection between these lattice paths and the balanced strings of parentheses which make up the Dyck language.  The correspondence can easily be seen by considering each north step as a left parenthesis and each east step as a right parenthesis.  At any point during a Dyck path walk, since we must remain at or above the line $y=x$, the number of north steps taken either equals or exceeds the number of east steps taken, although at the conclusion of the walk the totals are the same.  Therefore when the steps are recorded as parentheses, the string will begin with a left parenthesis and will not contain any unmatched parentheses.  So, sequences of balanced parentheses are an example of another Catalan-enumerated set and the paths which correspond to these sequences are described using the adjective `Dyck'.\newline
\text{ } 
\\
The figure below illustrates the Dyck paths for $n=1$ to $4$ (diagram modified from \cite{webDP}).
\begin{figure}[htbp]
	\centering
		\includegraphics[width=4.2in]{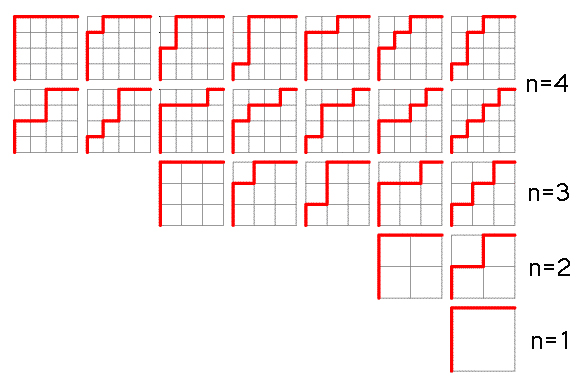}
	\caption{Dyck paths $n=1$ to $4$}
\end{figure}

\begin{definition}
A partially ordered set (or poset for short) is a set of objects, $P$, together with a $\leq$ relation on $P$ which satisfies three conditions:\newline
i) $\leq$ is reflexive ($x \leq x$, for all $x \in P$)\newline
ii) $\leq$ is transitive (if $x\leq y$ and $y \leq z$ then $x \leq z$ for all $x,y,z \in P$)\newline
iii) $\leq$ is antisymmetric (if $x \leq y$ and $y \leq x$, then $x=y$ for all $x,y \in P$).\newline
\end{definition}
We can impose a partial order on the set of Dyck paths by considering one path to be `less than' another if it lies below the other (see example below).  Considering any two Dyck paths of the same length, $d_1$ and $d_2$, we say that $d_1\leq d_2$ iff $d_1$ entirely lies below $d_2$, although the two paths may coincide at some points. In this paper, we refer to this partially ordered set as the poset of Dyck paths ordered by inclusion, although at times we may omit the phrase `ordered by inclusion'. We use $D_n$ to denote the Dyck path poset for paths of length (or order) $n$.\newline
\text{  }
\\
\begin{example}
Here are two Dyck paths from $D_4$.  The path on the left, $d_1$, lies below the path on the right, $d_2$, since they are of the same length, and $d_1$ does not cross over $d_2$ if the paths are superimposed.\newline
\begin{figure}[htbp]
	\centering
		\includegraphics[width=2in]{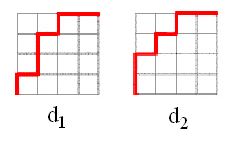}
	\caption{Path $d_1$ lies below path $d_2$}
	\label{fig:below}
\end{figure}
\newline
On the other hand, path $d_1$ does not lie below path $d_3$ since $d_1$ will cross over $d_3$ if the graphs are superimposed.  Similarly, $d_3$ does not lie below $d_1$.  As we shall see, this means that paths $d_1$ and $d_3$ are incomparable.
\begin{figure}[htbp]
	\centering
		\includegraphics[width=2in]{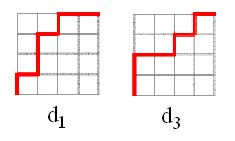}
	\caption{Path $d_1$ does not lie below path $d_3$}
	\label{fig:notbelow}
\end{figure}
\end{example}

\begin{definition}
Given a poset, $P$ and two elements $x,y \in P$, we say that $x$ is comparable to $y$ iff $x \leq y$ or $y \leq x$.  Otherwise we say that $x$ and $y$ are incomparable.
\end{definition}
\begin{definition}
An element $y$ covers (or is a cover for) an element $x$ in a poset, $P$, if $x<y$ (i.e. $x \leq y$, but $x \neq y$) and there is no other element, $z \in P$ where $x<z<y$.
\end{definition}%
\text{  }

It is common to use a Hasse diagram to provide a pictorial representation of the poset.  In the example below showing $D_3$ and $D_4$, the vertices are the Dyck paths and the edges illustrate the cover relationships that exist among the paths.  Also, the elements are grouped into ranks based on the area (or number of complete lattice cells) between each  Dyck path and the line $y=x$. As we will see in a later section, this area parameter will enable us to define a `$q$-analog' of the Catalan numbers.
\newpage
\begin{example}
The diagrams below show the Hasse diagrams for $D_3$ and $D_4$.
\begin{figure}[h]
	\centering
		\includegraphics[width=5in]{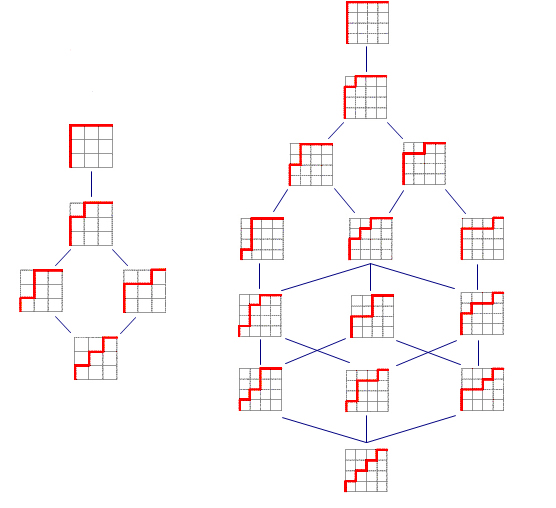}
	\caption{Hasse diagrams for $D_3$ (left) and $D_4$ (right)}
	\label{Dyck path poset for n = 3 and 4}
\end{figure} 
\end{example}

The purpose of this paper is to investigate some of the properties of the poset of Dyck paths ordered by inclusion.  Since this sequence is well-known as a way to count the total number of Dyck paths of order $n$, we begin by looking at a combinatorial explanation why the $C_n$ formula given above enumerates Dyck paths and then consider the recurrence relation and generating function for $C_n$.

\subsection{Combinatorial justification for enumerating $D_n$ using $C_n=\frac{1}{n+1}\binom{2n}{n}$}
\text{  }
\\
\\
1)  There are a total of $\binom{2n}{n}$ paths from $(0,0)$ to $(n,n)$ with north and east steps in an $n \times n$ grid, if these are the only directions in which we are permitted to travel.  We choose $n$ steps in the north direction and the remainder must travel east.  

2) Now, consider paths which cross below the line $y=x$ (i.e. not Dyck paths). There must be a first place where the path crosses below the main diagonal.  Take the portion of the path after the first `bad' step and interchange the north and east steps which is equivalent to  reflecting this portion of the path in the line $y = x - 1$ as in the figure below.

\begin{figure}[htbp]
	\centering
		\includegraphics[width=3in]{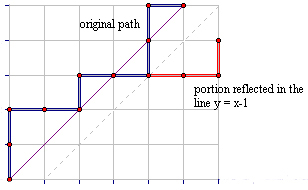}
	 \caption{Reflecting a Dyck path}
   \label{Reflecting a Dyck path}
\end{figure}

Now we have a situation where, for the $2n$ steps, $n+1$ are in one direction and $n-1$ are in the opposite direction.  In fact, this reflection maps the point $(x, y)$ to the point $(x',y')=(y+1,x-1)$ as shown geometrically below.  

\begin{figure}[htbp]
	\centering
		\includegraphics[width=3in]{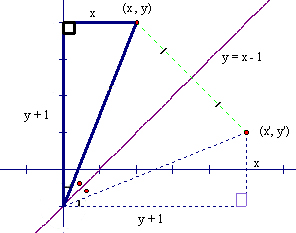}
	\caption{Reflecting in the line y = x - 1}
	\label{Reflecting in the line y = x - 1}
\end{figure}

The overall effect then is to take $(n,n) \rightarrow (n+1,n-1)$ or to create an $(n+1)$ x $(n-1)$ grid.  On this grid there are $\binom{2n}{n-1}$ ways of choosing a path. Since each path that crosses the diagonal can be uniquely transformed like this, there is a 1:1 correspondence.  So, the total number of such paths is given by $\binom{2n}{n-1}$.\newline
Thus the total number of Dyck paths, i.e. those which lie either entirely above or touch, but don't cross the $x$-axis, will be equal to the total number of paths $(=\binom{2n}{n})$ minus the number of paths which cross below the $x$-axis $(=\binom{2n}{n-1})$. 

\begin{align*}
\binom{2n}{n} - \binom{2n}{n-1} &= \frac{(2n)!}{n!n!} - \frac{(2n)!}{(n-1)!(n+1)!}\\
&= \frac{(2n)!}{n!(n-1)!}\left(\frac{1}{n} - \frac{1}{n+1}\right)\\
&= \frac{(2n)!}{n!(n-1)!}\left(\frac{(n+1)-n}{n(n+1)}\right)\\
&= \frac{1}{n+1}\left(\frac{(2n)!}{n!n!}\right)\\
&= \frac{1}{n+1} \binom{2n}{n}\\
\end{align*}
The above establishes the following proposition:
\begin{prop}
The number of Dyck paths of order $n$ in $D_n$ are counted by the Catalan number $C_n=\frac{1}{n+1}\binom{2n}{n}$.
\end{prop}
\text{ }
\\

\subsection{Recurrence relation and rotational transformation}
\text{  }
\\
\\
We know that the Catalan numbers $1,1,2,5,14,42,\ldots$ are defined:  $C_n=\frac{1}{n+1}\binom{2n}{n}$ and have established that this sequence counts the number of Dyck paths in the $D_n$ poset.
Consider the following recurrence relation:
\begin{equation*}
E_n=\sum_{k=1}^n E_{k-1}E_{n-k}\qquad \text{where } E_0 = E_1 = 1 
\end{equation*}
We will show that the total number of Dyck paths of order $n$, $\vert D_n\vert$, can also be enumerated using this recurrence relation and hence that $E_n$ is equivalent to the Catalan number, $C_n$.\newline 
\newpage
The table below lists the first five numbers calculated using this recurrence relation.  Note that these correspond to the Catalan numbers for $n=0$ to $5$.

\begin{tabular}[h]{l c r}
\scriptsize
n &  & \scriptsize $E_n$\\
\hline

\scriptsize 0 & \scriptsize$E_0 \text{ (by definition)}$ & \scriptsize1\\
\scriptsize 1 & \scriptsize$E_0E_0 = (1)(1) = 1$ & \scriptsize1\\
\scriptsize 2 & \scriptsize$E_0E_1 + E_1E_0 = (1)(1) + (1)(1) = 2$ & \scriptsize2\\
\scriptsize 3 & \scriptsize$E_0E_2 + E_1E_1 + E_2E_0= (1)(2) + (1)(1) + (2)(1)= 5$ & \scriptsize 5\\
\scriptsize 4 & \scriptsize$E_0E_3 + E_1E_2 + E_2E_1 + E_3E_0= (1)(5) + (1)(2) + (2)(1) + (5)(1)= 14$ & \scriptsize14\\
\scriptsize 5 & \scriptsize$E_0E_4 + E_1E_3 + E_2E_2 + E_3E_1 + E_4E_0= (1)(14) + (1)(5) + (2)(2) + (5)(1) + (14)(1) = 42$ & \scriptsize42 

\end{tabular}
\text{  }
\\

We now show that $E_n = C_n$ for all $n$ by providing a combinatorial explanation as to why this recurrence counts the number of Dyck paths.
It is easy to see that there is a bijection between Dyck paths and lattice paths starting at $(0,0)$ and ending at $(2n, 0)$ which never cross below the $x$-axis and consisting only of steps that are either NE $(1,1)$ or SE $(1,-1)$ since this would just involve rotating the Dyck path and the line $y=x$ by $45^\circ$  in a clockwise direction.
In this model, it is clear that the first step must be $(1,1)$ and the last step must be $(1, -1)$.  
The total number of steps $= 2n = n$ steps NE and $n$ steps SE.  We will call this a path of length (size) $n$.
For example, the poset $D_4$ is composed of paths of length 4 which have 4 steps NE and 4 steps SE.
If the path starts at $(0,0)$, it must come back down to the $x$-axis at some point.  We will consider $k$ to represent the SE step in which the path first reaches the $x$-axis.\newline
\begin{figure}[h]
	\centering
		\includegraphics[width=4.4in]{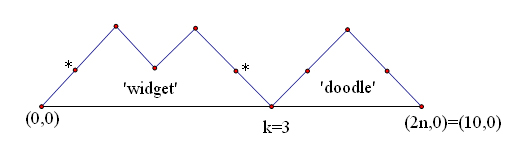}
		\caption{A Dyck path of length 5}
	\label{A possible Dyck path}
\end{figure}
\\
Note: we will use the term `widget' to denote the portion of the Dyck path from $(0,0)$ to $k$, and the term `doodle' to indicate the remainder of the Dyck path beyond $k$.\newline
In the diagram above we have a path of length 5 (i.e. $n = 5$) since there are 5 steps NE and 5 steps SE.  Here it first reaches the $x$-axis on the third SE step, so $k = 3$.  That means that there is a path of length 2 (= 2 SE steps) between the asterisks in the diagram above.  We know that the first step on the path must be in a NE direction and the last segment must be SE since there is only one way to do these, so in the widget we need only consider the path between the asterisks. For any Dyck path, the length of the portion between the asterisks will always be $k-1$.  Beyond $k$, there are $n-k$ NE and $n-k$ SE steps, for a path of length $n-k$ in the doodle.  By the multiplication principle, the number of paths for each $k$ will be equal to $E_{k-1}E_{n-k}$.  We note that $E_{k-1} E_{n-k}$ is equal to the number of pairs $(x, y )$ where $x$ is a widget of size $k - 1$ and $y$ is a doodle of size $n - k$ such that the size of the widget plus the size of the doodle is equal to $n-1$.
These correspond to the combinations in the centre column of the chart above, for example in the row where $n=5$, the pairs $(0,4), (1,3), (2,2), (3,1), (4,0)$ all sum to four.
However, the path can first reach the axis either when $k=1$ or when $k=2$, or \ldots, or when $k=n$ since the $k^{th}$ SE step can occur anywhere from the first to the $n^{th}$ SE step.  Since these cases are disjoint, we apply the addition principle to combine them.  So we consider (set of paths for $k=1$) + (set of paths for $k=2$) + \ldots + (set of paths for $k=n$). These will partition the total number of paths.\newline
The above discussion establishes the following proposition.
\begin{prop}
If we consider $E_0=1$ to be the empty path (since there is one way to do this), then the total number of paths, $\vert D_n \vert$, from $(0,0)$ to $(2n,0)$ is given by:
\begin{equation}
E_n=\sum_{k=1}^n E_{k-1}E_{n-k} \label{rrcomb}
\end{equation}
\end{prop}
\begin{remark}
Since both $C_n$ and $E_n$ count the number of Dyck paths of order $n$, it follows that $E_n = C_n$, so $E_n$ is a recurrence relation for the Catalan numbers.
\end{remark}
\subsection{Generating function}
\text{  }
\\
\\
Now we use the result of the previous section to prove algebraically that $E_n=C_n$ by considering the generating function for the Catalan numbers.\newline
Given the sequence of Catalan numbers as listed in the section above, we need to find the generating function for the infinite polynomial: 
\begin{equation*}1 + q + 2q^2 + 5q^3 + 14q^4 + 42q^5 + ...
\end{equation*}
Let \begin{equation*} E(q)=\sum_{n\geq0}E_nq^n \end{equation*} be a generating function for these numbers.  
Using the recurrence relation we have established for the Catalan numbers, it is evident that
\begin{equation*}(E(q))^2 = \sum_{n\geq0} \left(\sum_{k=0}^n E_{n-k} E_k\right)q^n =  
\sum_{n \geq 0} E_{n+1} q^n.\end{equation*}
Now, if we multiply both sides of this equation by $q$ and then rearrange the terms, we develop a quadratic equation in $E(q)$.
\begin{align*}
q(E(q))^2 &= q\sum_{n \geq 0}E_{n+1}q^n\\
&= \sum_{n \geq 0}E_{n+1}q^{n+1}\\
&= \sum_{n \geq 1}E_nq^n\\
&= \sum_{n \geq 0}E_nq^n - 1\\
&= E(q) - 1 \Rightarrow q(E(q))^2 - E(q) + 1 = 0\\
\end{align*}
Solving this using the quadratic formula, we obtain
\begin{equation*}
E(q)=\frac{1 \pm \sqrt{1-4q}}{2q}.
\end{equation*}
Noting that the derivatives of the numerator $1+\sqrt{1-4q}$ are negative, and given that our sequence has only positive terms, we reject this root.\newline
Therefore, the generating function can be rewritten in the following form:
\begin{align*}
E(q)&= \sum_{n \geq 0}E_nq^n = \frac{1-\sqrt{1-4q}}{2q}.\\
\end{align*}
Next we show by direct calculation,
\begin{equation*}
E(q)=\frac{1-\sqrt{1-4q}}{2q}=\sum_{n \geq 0}\frac{1}{n+1}\binom{2n}{n}q^n.
\end{equation*}
\begin{remark}
In order for us to work with the square root in the numerator of this expression, we shall extend the binomial theorem to include cases where the exponent is not a positive integer.  In \cite{MB}, the author expands the definition of binomial coefficients, $\binom{m}{k}$, so that $m$ can be any real number.  He defines $\binom{m}{0}=1$ and $$\binom{m}{k}=\frac{m(m-1)\ldots(m-k+1)}{k!}, \qquad \hbox{if $k \geq 0$}.$$
Then, using Taylor's theorem along with these definitions, he proves that, for any real number, $m$, and non-negative integer, $k$, $$(1+x)^m = \sum_{k \geq 0}\binom{m}{k}x^k.$$ 
\end{remark}
Applying this in our expression, we first note that
\begin{align*}
\binom{\frac{1}{2}}{n}&= \frac{\frac{1}{2}(\frac{-1}{2})(\frac{-3}{2})\ldots (\frac{3-2n}{2})}{n!}\\
&= \frac{(1)(-1)(-3)\ldots(3-2n)}{2^n\cdot n!}\\
&= \frac{(-1)^{n-1}(2n-3)(2n-5) \ldots(3)(1)}{2^n \cdot n!}\\
&= \frac{(-1)^{n-1}(2n-3)!_{odd}}{2^n\cdot n!}\\
\end{align*}
 where $(2n-3)!_{odd}=(2n-3)(2n-5)\ldots(3)(1)$.

Next, again using the result from \cite{MB},
\begin{align*}
\sqrt{1-4q} &= (1-4q)^{1/2}\\ &= \sum_{n \geq 0} \binom{1/2}{n}(-4q)^n\\ &= \sum_{n \geq 0}  \frac{(-1)^{n-1}(2n-3)!_{odd}}{2^n \cdot n!} \cdot(-4q)^n\\
&= \sum_{n \geq 0}  \frac{(-1)^{n-1}(2n-3)!_{odd}}{2^n \cdot n!}\cdot  (-1)^n \cdot (2^2)^n \cdot q^n\\
&= (-1) \cdot \sum_{n \geq 0} \frac{(2n-3)!_{odd}}{n!}\cdot 2^n \cdot q^n\\
&= -\sum_{n \geq 0}\frac{(2n-3)!_{odd}\cdot n! \cdot 2^n}{n!\text{ }n!}q^n\\
&= -\sum_{n \geq 0}\frac{(2n-3)!_{odd}\cdot n \cdot 2 \cdot (2 \cdot 4 \cdot 6 \cdot \ldots \cdot 2n-2)}{n!\text{ }n!}q^n\\
&= -2 \sum_{n \geq 0} \frac{(2n-2)!}{n!(n-1)!}q^n\\
&= -2 \sum_{n \geq 0} \frac{\binom{2n-2}{n-1}}{n}q^n.
\end{align*}

Now the generating function for the Catalan numbers can be rewritten as follows:
\begin{align*}
E(q) =\frac{1-\sqrt{1-4q}}{2q} &= \frac{1+2(1+\sum_{n \geq 1}\frac{\binom{2n-2}{n-1}}{n}q^n)}{2q}\\
&= \frac{3+2\sum_{n \geq 1}\frac{\binom{2n-2}{n-1}}{n}q^n}{2q}.\\
\end{align*}

The coefficient, $E_n$, of $q^n$ occurs when $n=N+1$.  Hence,
\begin{align*}
E_n &= \frac{2\frac{\binom{2(N+1)-2}{N+1-1}}{N+1}}{2}\\
&= \frac{\binom{2n}{n}}{n+1}\\
&= \frac{1}{n+1}\binom{2n}{n}.\\
\end{align*}
\newline
Since we recognize this formula as being the one for the Catalan numbers, we have just arrived at the proposition below using a second approach:
\begin{prop}
$E_n=C_n$, where $C_n = \frac{1}{n+1}\binom{2n}{n}$ for $n \geq 0$.
\end{prop}

\text{ }

\subsection{The main result}
\text{  }

When considered together, the results from the previous sections lead to the following theorem:
\begin{theorem}
The number of Dyck paths of order $n$ in the poset $D_n$,
\begin{align}
\vert D_n \vert &= C_n \\
&= \frac{1}{n+1}\binom{2n}{n}\text{ for }n \geq 0 \nonumber\\
&= \sum_{k=1}^n C_{k-1}C_{n-k} \text{ where } C_0=C_1=1.\nonumber\\
\nonumber
\end{align}
\end{theorem}

\newpage
\section{$q$-analogs for the Catalan numbers}
\subsection{Classical $q$-analogs}
\text{  }
\\
\\
Next we consider $q$-analogs for the Catalan numbers.  These enable us to preserve more information about $D_n$ within the sequences that enumerate the poset.  We begin with the three classical $q$-analogs, $C_n^{area}(q)$, $C_n^{inv}(q)$   and $C_n^{maj}(q)$.
\begin{definition}
The area of a Dyck path, $area(D)$, is the number of complete lattice cells lying between the Dyck path and the main diagonal.
\end{definition}
\begin{definition}
If $D$ is a Dyck path of length $n$, then the area q-analog of the Catalan numbers,
$$C_n^{area}(q)=\sum_{D \in D_n}q^{area(D)}$$
\end{definition}
This leads to a polynomial in $q$ in which the Dyck path count is decomposed into terms where the degree of the exponent indicates the area of the Dyck path(s) and the coefficient in front specifies the number of Dyck paths covering that particular area.  
\begin{example}
Consider the five paths in $D_3$.  Since there is one path with $area=0$, two paths with $area=1$, one path with $area=2$ and one path with $area=3$, we compute
$$C_3^{area}(q)= 1 + 2q + q^2 + q^3$$ and note that by setting $q=1$ we obtain $C_3^{area}(1)=5=C_3$.
\end{example}
\begin{figure}[htbp]
	\centering
		\includegraphics[width=5in]{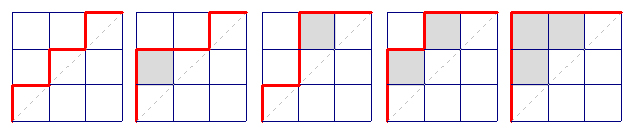}
	\caption{Area under Dyck paths of length 3}
	\label{Area under Dyck paths of length 3}
\end{figure}

Closely related to the area $q$-analog is the inversion $q$-analog.  
\begin{definition} We define inversion(D) or $inv(D)$ to be the area within the grid that lies above the Dyck path. 
\end{definition}
Since the $\binom{n}{2}$ lattice cells above the main diagonal in an $n \times n$ grid must either lie above the Dyck path or below it, and recalling that $area(D)$ counts the cells lying below the Dyck path,  it follows that $inv(D) = \binom{n}{2} - area(D)$. Using this relationship we can now define an inversion $q$-analog of the Catalan numbers:\newline
\begin{align}
C_n^{inv}(q) &= \sum_{D \in D_n} q^{inv(D)}\nonumber\\
&= \sum_{D \in D_n} q^{\binom{n}{2}-area(D)}\nonumber\\
&= q^{\binom{n}{2}}\sum_{D \in D_n}(1/q)^{area(D)}\nonumber\\
&= q^{\binom{n}{2}}C_n^{area}(1/q)\label{cninvarea}\\
\nonumber \end{align}
\begin{example}
For the Dyck paths of length 3, there is one path with $inv(D)=0$, one with $inv(D)=1$, one with $inv(D)=3$, and two paths with $inv(D)=2$, the corresponding $q$-analog, $C_3^{inv}(q) = 1 + q + 2q^2 + q^3$. \newline
\begin{figure}[htbp]
	\centering
		\includegraphics[width=5in]{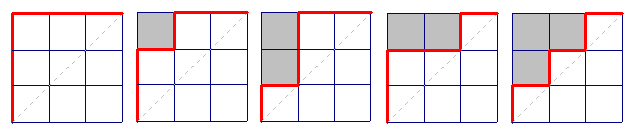}
	\caption{Inversion $q$-analog for Dyck paths of length 3}
	\label{  }
\end{figure}
\end{example}
\begin{remark} As with other classical $q$-analogs of the Catalan numbers, if we set $q=1$ then we recover $C_n$.
\end{remark}

Now, suppose that we number consecutive segments of a Dyck path $1, 2, \ldots, 2n$ starting with the initial north step up from $(0,0)$.
\begin{definition}
Given such a numbering, the major index of a Dyck path, $maj(D)$, is equal to the sum of the integers on any east steps that immediately precede a north step along the path, $D$.
\end{definition}
\begin{definition}
If $D$ is a Dyck path of length $n$, then the major index $q$-analog of the Catalan numbers,
$$C_n^{maj}(q)=\sum_{D \in D_n} q^{maj(D)}.$$
\end{definition}
As with the area $q$-analog the exponents on this polynomial group the Dyck paths, this time according to their major indices, while the coefficients indicate the number of Dyck paths in each group.\newline
\begin{example}
Again examining the five paths in $D_3$, we compute the major index for each path and construct $$C_n^{maj}(q)= 1 + q^2 + q^3 + q^4 + q^6.$$
\end{example} 
\begin{figure}[htbp]
	\centering
		\includegraphics[width=5in]{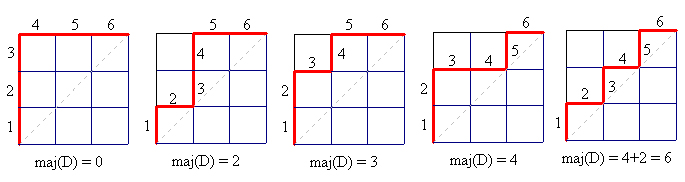}
	\caption{Major index for Dyck paths of length 3}
	\label{Major index for Dyck paths of order 3}
\end{figure}

Notice once again that if we set $q=1$ in the polynomial, we recover $C_3 = 5$.  Furthermore, as proved in \cite{Mac} there is an explicit formula for computing $C_n^{maj}$ involving the $q$-binomial numbers.  If we define 
\begin{align*}
[n]_q &= \frac{1-q^n}{1-q} = 1 + q + q^2 + \ldots + q^{n-1}, \text{ for } n \geq 0\\
\text{ }\\
\text {and }\qquad\begin{bmatrix} n\\k \end{bmatrix}_q &= \frac{[n]!}{[k]![n-k]!}\\
\text{ }\\
\text {where }\qquad[n]! &= [n][n-1][n-2] \ldots [2][1]\\
\end{align*}
then 
\begin{equation}
\label{cnmaj} C_n^{maj}(q) = \frac{1}{[n+1]_q}\begin{bmatrix}2n\\n \end{bmatrix}_q.\\
\end{equation}

\subsection{Bivariate Catalan sequences}
\text{  }
\\
\\
Remarkably the two $q$-analogs $C_n^{area}(q)$ and $C_n^{maj}(q)$ are related to one another.  In order to demonstrate this relationship, we need to define the bivariate Catalan sequence $C_n(q,t)$.  Stemming from their work with symmetric polynomials and representation theory, Garsia and Haiman introduced the $C_n(q,t)$ sequence along with a formula (see \cite{AGMH} for a proof)
\begin{equation}
C_n(q,t)=\sum_{\mu \vdash n}\frac{t^{2\sum l}q^{2\sum a}(1-t)(1-q)\prod^{0,0}(1-q^{a'}t^{l'})\sum q^{a'}t^{l'}}{\prod (q^a - t^{l+1})(t^l-q^{a+1})}\label{bigformula}
\end{equation}
which is a sum over all partitions of $n$.  Within the formula, sums and products are over all cells in the partition $\mu$ and for each cell $c \in \mu$, the variables $l, l', a, a'$ indicate the number of cells that lie strictly above, below, to the right of and to the left of $c$.  In order that the product in the numerator does not equal zero, the cell with $a'(c)=l'(c)=0$ must be omitted from the calculation.  This is denoted by the symbol $\prod^{0,0}$. \newline

\begin{example}
Consider the partitions of three.  The corresponding values of $a, a', l,$ and $l'$ for each of the cells in the  partitions are indicated in the figure below.
\begin{figure}[htbp]
	\centering
		\includegraphics[width=3.2in]{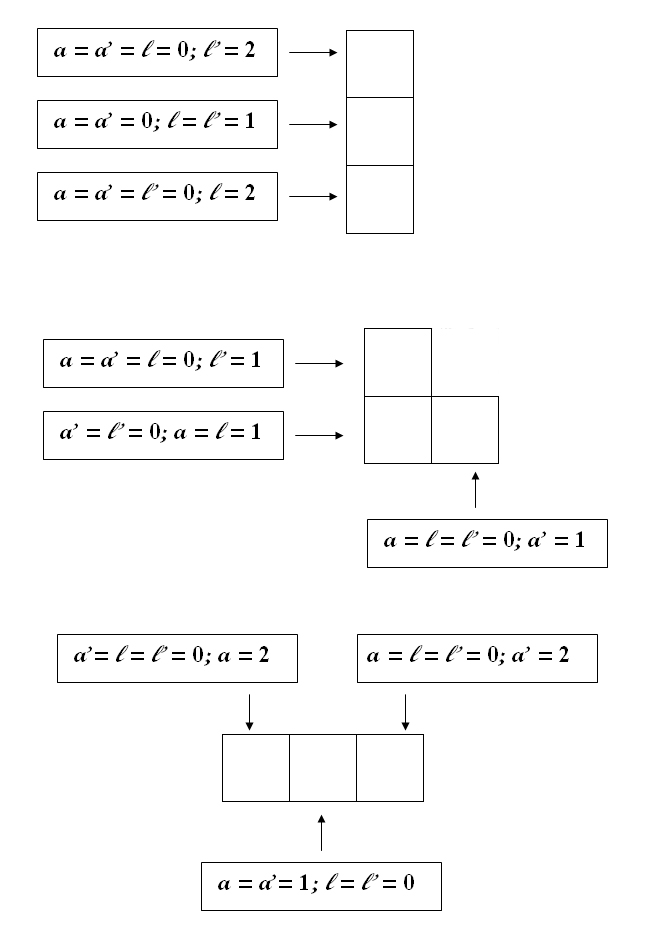}
	\caption{Partitions, $\mu$, for $n=3$}
	\label{fig:ptnsof3}
\end{figure}
\end{example}
When these values are substituted into equation \eqref{bigformula}, we obtain
\begin{align*}
C_3(q,t)&=\frac{t^6q^0(1-t)(1-q)(1-t)(1-t^2)(1+t+t^2)}{(1-t^3)(t^2-q)(1-t^2)(t-q)(1-t)(1-q)}\\
& + \frac{t^2q^2(1-t)(1-q)(1-t)(1-q)(1+t+q)}{(q-t^2)(t-q^2)(1-t)(1-q)(1-t)(1-q)}\\
& + \frac{t^0q^6(1-t)(1-q)(1-q)(1-q^2)(1+q+q^2)}{(q^2-t)(1-q^3)(q-t)(1-q^2)(1-t)(1-q)}.\\
\end{align*}
While this sum may appear to be just a messy rational expression in $q$ and $t$, it turns out that it simplifies to the polynomial $q^3 + tq^2 + t^3 + qt + qt^2$, in which the Catalan number, $C_3$ can be recovered by setting $q = t = 1$.  In addition, with this example, we see that
\begin{align*}
C_3(q,1)&= q^3 + q^2 + 1 + q + q\\
& = 1 + 2q + q^2 + q^3\\
& = C_3^{area}(q)\\
\end{align*}
and 
\begin{align*}
q^{\binom{3}{2}}C_3(q,1/q)&= q^3(q^3 + (1/q)(q^2) + (1/q)^3 + q(1/q) + q(1/q)^2)\\
& = q^3(q^3 + q + (1/q)^3 + 1 + (1/q))\\
& = 1 + q^2 + q^3 + q^4 + q^6\\
& = C_3^{maj}(q).\\
\end{align*}

Garsia and Haiman were able to prove several specializations of equation \eqref{bigformula} (for all $n$) and by considering these results it becomes evident why they named $C_n(q,t)$ the $q,t$-Catalan sequence.  Note how the $q$-analogs discussed earlier can both be derived from $C_n(q,t)$. (See \cite{AGMH} for proofs of the following equalities.)
\begin{align*}
C_n(1,1)&=\frac{1}{n+1}\binom{2n}{n} = C_n\\
q^{\binom{n}{2}}C_n(q,1/q) &= \frac{1}{[n+1]_q}\begin{bmatrix}2n\\n \end{bmatrix}_q = \sum_{D \in D_n}q^{maj(D)}= C_n^{maj}(q)\\
C_n(q,1)&=\sum_{D \in D_n}q^{area(D)}=C_n^{area}(q)\\
\end{align*}

In equation \eqref{bigformula} it is clear that $C_n(q,t)$ is symmetric, that is we are able to interchange the variables $q$ and $t$ and still end up with the same result.  However, despite the fact that $C_n(q,t)= C_n(t,q)$ is straightforward to show algebraically, finding a combinatorial proof remains an open problem \cite{NL3}.\newline
\text{  }
\\
When enumerating Dyck paths, the $q,t$-Catalan sequences have the form
\begin{equation}
C_n(q,t)=\sum_{D \in D_n} q^{area(D)} t^{tstat(D)}\label{qareatstat}
\end{equation}
where the exponent on $q$ is the area statistic discussed previously and the exponent on $t$ is a $t$-statistic.  Two different combinatorial interpretations have been proposed for the $t$ statistic.  The first of these, called the bounce statistic, was introduced by Haglund \cite{JH}.
\begin{definition} 
Let $D$ be a Dyck path of order $n$.  Define a bounce path (derived from $D$) as a sequence of alternating west and south steps beginning at $(n,n)$ and proceeding to $(0,0)$ according to the following steps:  Starting at $(n,n)$, travel west until the last north segment of the Dyck path is reached. Then travel south until the main diagonal (the line $y=x$) is hit.  Travel west from this point until the Dyck path is met again, and then south to the main diagonal. Repeat this step as many times as necessary in order that the path terminates at $(0,0)$.
\end{definition}
\begin{definition} Given a bounce path derived from $D$ as defined above, assign the numbers, $k=0,1,2,\ldots,n$ to the lattice points $(k,k)$ along $y=x$. Then the bounce statistic, $bounce(D)$, is equal to the sum of the $k$ values corresponding to points where the bounce path meets the line $y=x$, excluding the starting point $(n,n)$.
\end{definition}
\begin{example}
A Dyck path, $D$ is shown in the figure below, along with its derived bounce path.  Since the bounce path meets the main diagonal at $(0,0),(1,1)$ and $(3,3)$ if we exclude the starting point, then $bounce(D)=1+3=4$.
\begin{figure}[htbp]
	\centering
		\includegraphics[width=2.5in]{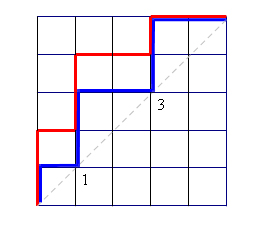}
	\caption{A Dyck path (red) and its derived bounce path (blue)}
	\label{A Dyck path and its derived bounce path}
\end{figure}
\end{example}

A second combinatorial interpretation for the $t$-statistic was suggested by Haiman, but this has been shown to be equivalent to Haglund's bounce statistic \cite{NL2}.  Therefore, we can rewrite equation \eqref{qareatstat} as
\begin{equation}
C_n(q,t)=\sum_{D \in D_n} q^{area(D)}t^{bounce(D)}.\label{qareatbounce}
\end{equation}
\text{ }

Recall that if we set $t=1$ in \eqref{qareatbounce}, then
\begin{equation}
 C_n(q,1)=\sum_{D \in D_n} q^{area(D)}.\label{qareasum}
\end{equation}
\text{ }

Now if we set $q=1$ in \eqref{qareatbounce}, then
\begin{equation}
C_n(1,t)= \sum_{D \in D_n} t^{bounce(D)}.\label{tbouncesum}\\
\end{equation}
\text{ }

Then we can set $t=q$ in \eqref{tbouncesum}, which will yield 
\begin{equation*}
C_n(1,q)=\sum_{D \in D_n} q^{bounce(D)}.
\end{equation*} 
\text{ }

Since $C_n(q,t)$ is symmetric, $C_n(q,1)=C_n(1,q)$, and we arrive at the following proposition:\newline
\begin{prop}
\begin{equation*}
C_n^{area}(q)=\sum_{D \in D_n} q^{area(D)} = \sum_{D \in D_n} q^{bounce(D)}.
\end{equation*}
\end{prop}
To emphasize, this means that if we specialize \eqref{qareatbounce} by setting $q$ or $t=1$, we can obtain $C_n$ whether we are summing over Dyck paths by area or by bounce. \newline 
\text{ }
\\
In fact, in \cite{NL2}, as part of the proof that the two conjectured interpretations for the $t$-statistic are equivalent, the author constructs a bijection $\phi:D_n \rightarrow D_n$ having the property that $bounce(\phi(D))=area(D)$.  However, the details of this construction and accompanying proof are beyond the scope of this paper.
\newpage
The table below lists the $q,t$-Catalan numbers for $n=0$ to $5$.  Recall that these can be computed using equation \eqref{bigformula}, where remarkably, the complicated rational functions which are produced by the equation always simplify to polynomials.
\text{  }
\\
\begin{center}
\begin{tabular}[h]{l l}

\scriptsize
n &  \qquad \qquad \qquad \qquad \qquad \scriptsize $C_n(q,t)$ \\
\hline

\scriptsize 0 & \qquad \scriptsize 1\\
\scriptsize 1 & \qquad \scriptsize 1\\
\scriptsize 2 & \qquad \scriptsize $q + t$\\
\scriptsize 3 & \qquad \scriptsize $q^3 + q^2t + qt^2 + qt + t^3$\\
\scriptsize 4 & \qquad \scriptsize $q^6 + q^5t + q^4t + q^4t^2 + q^3t + q^3t^2 + q^3t^3 + q^2t^2 + q^2t^3 +$\\
\scriptsize {  } & \qquad \scriptsize $q^2t^4 + qt^3 + qt^4 + qt^5 + t^6$\\
\scriptsize 5 & \qquad \scriptsize $q^{10} + q^9t + q^8t + q^8t^2 + q^7t + q^7t^2 + q^7t^3 + q^6t + 2q^6t^2 + q^6t^3 +$\\
\scriptsize {  } & \qquad \scriptsize $q^6t^4 + q^5t^2 + 2q^5t^3 + q^5t^4 + q^5t^5 + q^4t^2 + 2q^4t^3 + 2q^4t^4 + q^4t^5 + $\\
\scriptsize {  } & \qquad \scriptsize $q^4t^6 + q^3t^3 + 2q^3t^4 + 2q^3t^5 + q^3t^6 + q^3t^7 + q^2t^4 + q^2t^5 + 2q^2t^6 + $\\
\scriptsize {  } & \qquad \scriptsize $q^2t^7 + q^2t^8 + qt^6 + qt^7 + qt^8 + qt^9 + t^{10}$\\
\end{tabular}
\end{center}

\newpage
\section{Incidence algebra and M\"obius function}
\subsection{Incidence algebra}
\text{  }
\\
\\
Let $P$ be a finite poset and $x,y \in P$.  We consider the incidence algebra of two parameter functions on $P$ such that $f(x,y)=0$ if $x \nleq y$ and $f:P \times P \rightarrow \mathbb{R}$.  The zeta and M\"obius functions defined below are examples of these functions.  In this algebra, addition and scalar multiplication are defined on a pointwise basis and the multiplication operation is a convolution defined by:
\begin{equation*}(f*g)(x,y)=\sum_{x \leq z \leq y} f(x,z)g(z,y) \end{equation*}
But this definition is analogous to ordinary matrix multiplication.
Therefore, if we consider matrices, $M$, whose rows and columns are indexed by elements of $P$, then the incidence algebra of matrices on $P$ consists of all matrices $M$ such that $M(x,y)=0$ unless $x \leq y$ in $P$.
So, if we take two matrices $(M, N)$ from the incidence algebra of $P$, that is such that $M_{xy}=0$ and $N_{xy}=0$ if $x\nleq y$, and perform matrix multiplication it is clear that
\begin{equation*} (MN)(x,y)=\sum_{z\in P}M(x,z)N(z,y) \end{equation*}
will extend over only those values of $z$ which are contained in the interval $x\leq z \leq y$.
Note that these algebras of functions on $P$ and matrices on $P$ are isomorphic, hence they record the same information.  In general, the incidence algebra of a poset $P$, denoted $I(P)$, is an algebra with a basis $[x,y]$ where $[x,y]$ is an interval of $P$.
\begin{definition}
Let $P$ be a poset.  If $[x,y] = \lbrace z: x \leq z \leq y \text{ for } z \in P \rbrace$, then $[x,y]$ is called an interval of the poset $P$.
\end{definition}
By calculating the dimension of $I(P)$ for the Dyck path poset with $n$ values from $1$ to $5$, we found that the number of intervals in $D_n$, are as shown in the table below.
\begin{center}
\begin{tabular}[h]{l c r}

\scriptsize
n &  \scriptsize Dimension of $I(P)$ \\
\hline

\scriptsize 1 & \scriptsize 1\\
\scriptsize 2 & \scriptsize 3\\
\scriptsize 3 & \scriptsize 14\\
\scriptsize 4 & \scriptsize 84\\
\scriptsize 5 & \scriptsize 594\\
\end{tabular}
\end{center}
\text{}
\\
\noindent This sequence appears in \cite{OLEIS} as A005700, the number of paths of length $2n$ in the first octant of the $x,y$-plane which start and end at the origin.  Therefore there must be a correspondence between intervals in the Dyck path poset and paths of this type.
\newpage
\subsection{Zeta and M\"obius functions}
\text{  }
Consider specifically the element of the incidence algebra of $P$, which is defined as follows:
\begin{equation*} \zeta(x,y)=\left\{
\begin{array}{lr}
1 & \text{if }x \leq y \text{  in } P\\
0 & \text{otherwise }
\end{array} \right.
\end{equation*}
Known as the zeta function of $P$, this forms an upper triangular matrix with ones along the main diagonal since all elements are equal to themselves. It also has ones across the top row since the element at the `base' of the poset diagram is comparable and less than or equal to all other elements of the poset.   For any upper triangular matrix with ones along the main diagonal, the determinant is equal to 1.  Since the determinant is non-zero, we know that the matrix $\zeta$ is invertible.  Since $\zeta$ is invertible, an inverse made up of the product of elementary matrices exists.  But each of these elementary matrices will also be upper triangular.  Therefore the product of the elementary matrices (the inverse of $\zeta$) is also an upper triangular matrix. We call this inverse the M\"obius function of $P$ and it is denoted by $\mu (P)$.  Note that $\mu (x,x) = 1$ which follows from the upper triangularity of the matrix.
Since $\mu\zeta=\delta$, where $\delta$ denotes the identity, and the determinants of both $\zeta$ and $\delta$ are already known to be 1, it follows that the determinant of $\mu(P)$ is also equal to 1. In addition, we claim that the elements of $\mu(P)$ must be 1, 0 or -1. 
\begin{equation*} \mu\zeta=(\zeta)^{-1}\zeta=\frac{Adj(\zeta)}{\det(\zeta)}*\zeta=\delta
\end{equation*}
So, $\mu(D)= Adj(\zeta)/\det(\zeta)$.  As stated previously, the determinant of $\zeta$, is 1.  Each element in the numerator, $Adj(\zeta)$ is a determinant, which comes from the original matrix by considering deletions of rows and columns.  In deleting rows and columns from the upper triangular matrix $\zeta$ which consisted of only 1's and 0's, we will create only upper triangular matrices whose determinants will either be 1, 0 or -1 depending on which row and column are deleted. Essentially there are three cases since the intersection of the row and column is either above, on or below the main diagonal.  Thus, $\zeta^{-1}$ is determined solely by $Adj(\zeta)$.  Since these determinants form the elements of $Adj(\zeta)$, and dividing by 1 doesn't change these values, it follows that the only possible values for $\mu(P)$ are 1, 0 and -1.
\begin{definition}\label{mobius}
We define $\mu(P)$ such that: $\mu(x,x):=1$, $\mu(x,y):=0$ if $x \nleq y$, and $$\mu(x,y):=-\sum_{x \leq z < y} \mu(x,z) \qquad \text{   for   } x < y \text{ in } P.$$
\end{definition}
\text{ }
\\
Since $D_n$ is a finite poset, the above discussion and equations are relevant.
\text{ }
\\
\begin{example}Consider the poset of Dyck paths ordered by inclusion for $n = 3$.  The Hasse diagram is drawn below with the vertices representing Dyck paths and the edges showing the cover relationships.  The vertices are labelled here in order to facilitate the discussion.  
We can compute the values by hand using definition \ref{mobius}.
\begin{figure}[htbp]
	\centering
		\includegraphics[width=1.7in]{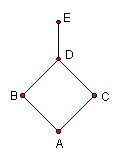}
	\caption{Dyck path poset for n=3}
	\label{Dyck path poset for n=3}
\end{figure}
For example, starting at the base of the diagram, we know that $\mu(A,A)=1$ by definition.  Then to find $\mu(A,B)$ , we use equation  where $\mu(A,A)=-\mu(A,B)$.  It follows that $\mu(A,B)=-1$.  By a similar argument, $\mu(A,C)=-1$.  Then $\mu(A,D)=-(\mu(A,A) + \mu(A,B)+\mu(A,C))=-(1+(-1)+(-1))=1$ and $\mu(A,E) =-(\mu(A,A)+\mu(A,B)+\mu(A,C)+\mu(A,D))=-(1 + (-1) + (-1) + 1) = 0$.  By a similar procedure it is possible to compute the M\"obius function value for any pair of elements in the poset.
\end{example}

We can also use Maple to quickly find the M\"obius function values for all pairs $(x, y) \in D_3$.
We enter the zeta matrix for $D_3$, where the $(i,j)$th entry is 1 if $i\leq j$ in $D_3$, and 0 otherwise, in Maple. Then, using the MatrixInverse function in Maple, we invert this matrix and obtain the M\"obius matrix for $D_3$.\newline
\begin{figure}[h]
	\centering
	\begin{equation*}
	\begin{bmatrix}
	1 & 1 & 1 & 1 & 1\\
	0 & 1 & 0 & 1 & 1\\
	0 & 0 & 1 & 1 & 1\\
	0 & 0 & 0 & 1 & 1\\
	0 & 0 & 0 & 0 & 1
	\end{bmatrix}
			\end{equation*}
			\caption{Zeta function matrix for $D_3$}
			\label{Zeta function matrix for $D_3$}
\begin{equation*}
	\begin{bmatrix}
	1 & -1 & -1 & 1 & 0\\
	0 & 1 & 0 & -1 & 0\\
	0 & 0 & 1 & -1 & 0\\
	0 & 0 & 0 & 1 & -1\\
	0 & 0 & 0 & 0 & 1
	\end{bmatrix}
			\end{equation*}			
			\caption{M\"obius function matrix for $D_3$}
	\end{figure}

%

However, both of these methods are time-consuming for larger values of $n$.  In order to find the M\"obius function value for any given pair of elements in the poset, it will be useful to have a direct formula.
\subsection{Isomorphism between $D_n$ and $J(P)$}
\text{  }
\\
\\
\noindent In order to find a formula for the M\"obius function of the poset of Dyck paths, we must first prove that the poset of Dyck paths ordered by inclusion is isomorphic to the poset of order ideals ordered by inclusion of the poset of points $(a,b)$ where $a+b \leq (n-2)$ and $a,b\geq 0$ with $(a,b) \leq (c,d)$ iff $a \geq c$ and $b \geq d$.\newline
\text{  }
\\
However, it will be helpful to recall some definitions and consider examples before proving this claim.
\text{  }
\\
Recall that a Dyck path is a lattice path in the $n \times n$ square which doesn't pass below the line $y=x$.  A walk along a Dyck path consists of $2n$ steps, with $n$ in the north direction and $n$ in the east direction.  By necessity the first step must be north and the final step must be east.
\begin{example}
Recall that there are the 14 possible Dyck paths for $n=4$ (diagram modified from \cite{webDP}).\newline

\begin{figure}[h]
	\centering
		\includegraphics[width=5in]{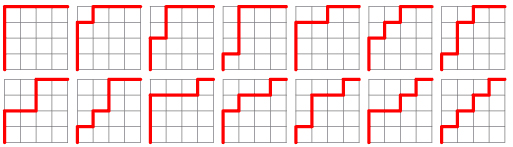}
	\caption{Dyck paths for n = 4}
\end{figure}
\end{example}
\text{  }
\\
If the Dyck paths are ordered by inclusion, i.e. one path is less than another if it lies completely below, then we form a poset.  Suppose we have two Dyck paths, $x$ and $y$ such that $x < y$. If there is no path $z$ in the poset such that $x < z <y$ then we say that $y$ covers $x$.  The poset can be displayed using a Hasse diagram in which the vertices represent the Dyck paths and the edges show the cover relationships.  Here again is the Hasse diagram for the Dyck path poset where $n=4$.\newpage
\begin{figure}[h]
	\centering
		\includegraphics[width=1.7in]{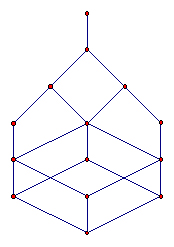}
	\caption{Dyck path poset for n = 4}
\end{figure}
\begin{example}
Consider the poset $P_n$, consisting of points $(a, b)$ where $a+b \leq (n-2)$ and $a,b\geq 0$ with $(a,b) \leq (c,d)$ iff $a \geq c$ and $b \geq d$.  If $n=4$ then $a+b \leq (n-2)$ implies that $(a+b)\leq 2$ so the points are $(0,0), (0,1), (1,0), (1,1), (0,2), (2,0)$.  Imposing the order outlined above we obtain the following diagram:
\end{example}
\begin{figure}[h]
	\centering
		\includegraphics[width=2.75in]{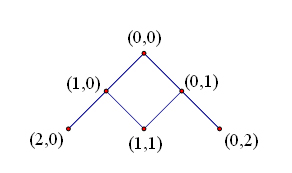}
	\caption{Poset of points $P_n=\{(a,b): a+b\leq (n-2)\}$ for $n = 4$ where $(a,b)\leq(c,d)$ iff $a \geq c$ and $b \geq d$}
	\label{ }
\end{figure}
\text{ }
\\
\begin{definition}
An order ideal (or downset) of $P$ is a subset, $I$ of $P$ such that if $x \in I$ and $y \leq x$, then $y \in I$.\end{definition}

\text{  }

It is easy to prove that the set of all order ideals ordered by inclusion forms a poset, which we shall denote $J(P)$.  It is this poset, $J(P)$, which we claim is isomorphic to the Dyck path poset.  For our example, the order ideals are:
\begin{center}
$\{ \text{ } \}$, $\{(2,0)\}$, $\{(1,1)\}$, $\{(0,2)\}$, $\{(2,0),(1,1)\}$, $\{(1,1),(0,2)\}$, $\{(2,0),(0,2)\}$, 
$\{(1,0),(2,0),(1,1)\}$, $\{(2,0),(1,1),(0,2)\}$, $\{(0,1),(1,1),(0,2)\}$, $\{(1,0),(2,0),(1,1),(0,2)\}$,   $\{(0,1),(1,1),(0,2),(2,0)\}$,  $\{(1,0),(0,1),(2,0),(1,1),(0,2)\}$, $\{(0,0),(1,0),(0,1),(2,0),(1,1),(0,2)\}$.
\end{center}

If these are ordered by inclusion, the following Hasse diagram is formed:
\begin{figure}[h]
	\centering
		\includegraphics[width=3.5in]{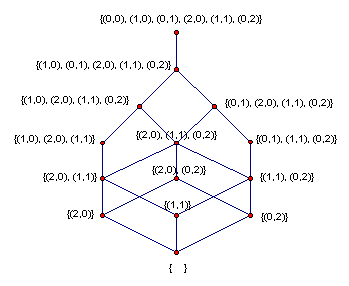}
	\caption{Order ideals poset for n = 4}
\end{figure}
\text{ }
\\
We note that for our example with $n=4$, the Hasse diagrams are identical for the poset of Dyck paths and for the poset $J(P)$ of order ideals of $P_n$ where both posets are ordered by inclusion.\newline
We are now ready to prove the claim.
\begin{prop}
The poset of Dyck paths ordered by inclusion is isomorphic to the poset of order ideals ordered by inclusion of the poset of points $(a,b)$ where $a+b \leq (n-2)$ and $a,b\geq 0$ with $(a,b) \leq (c,d)$ iff $a \geq c$ and $b \geq d$.
\end{prop}
\begin{proof}
Let $D$ be the poset of Dyck paths (ordered by inclusion) for an $n \times n$ grid.  Let $P_n$ be the poset of points $(a,b)$ where $(a+b) \leq (n-2)$ and $a,b \geq 0$ with $(a,b) \leq (c,d)$ iff $a \geq c$ and $b \geq d$.  Let $J(P)$ be the poset of order ideals of $P_n$. Consider $P_n$ as a staircase partition where each lattice point in $P_n$ corresponds to a cell in the partition.  There are $\binom{n}{2}$ lattice points in $P_n$ which will be arranged in the partition such that the top row is of size $(n-1)$ and in each row below the size is one cell smaller than in the previous row, thus creating the staircase shape.  The partition cells are labelled so that cells which lie below and to the right reflect the cover relationships from $P_n$. The diagram below illustrates the labelled partition diagram up to $n=4$.\newline
\begin{figure}[h]
	\centering
		\includegraphics[width=2.5in]{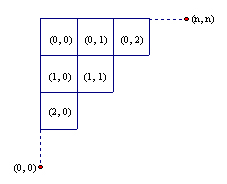}
	\caption{Staircase partition for n = 4}
	\label{Staircase partition for n = 4}
\end{figure}
\text{ }
\\
\noindent Now each order ideal from $J(P)$ will correspond to a path through the partition drawn from $(0,0)$ to $(n,n)$ using only north and east steps (making Dyck paths) and such that if a point is included in the order ideal set, then it lies below the path in the partition diagram. \newline
Let $f:J(P) \rightarrow D$ be defined as follows:  Let $I \in J(P)$ be the order ideal $\{(a_1, b_1), (a_2, b_2), \ldots, (a_k, b_k)\}$.  By the definition of order ideal, if an element $(a_i, b_i)$ is included in the set, then all $(a_j, b_j)\leq (a_i, b_i)$ must also be included in the set.  Then $f(I)$ is the path from $(0,0)$ to $(n,n)$ such that $(a_1, b_1), (a_2, b_2), \ldots,(a_k, b_k)$ are the cells which lie below the path in the partition. This path will be a Dyck path since:\newline
 a) It won't pass below the straight line joining $(0,0)$ to $(n,n)$.  The minimal possible path, corresponding to the order ideal $\{ \text{ } \}$, consists of steps following the staircase pattern $(NE)^n = NENE\ldots NE$ which will trace along the bottom and right borders of the cells $(a, b)$ where $(a+b)=(n-2)$. \newline
 b)  It won't pass outside of the partition along the left or top borders.  The maximal possible path, corresponding to the order ideal containing all elements of the poset, $P_n$, will contain all labelled cells in the partition and will trace along the left and top borders of the partition diagram, along the cells $(a, b)$ where at least one of $a,b$ is 0.\newline
 c)  The path can be constructed using only north and east steps.  All other paths corresponding to order ideals will lie above (or contain) the minimal path, and will lie below (or be contained in) the maximal path.  If a path lies above cell $(a_i, b_i)$ then it will also lie above all $(a_j, b_j)\leq (a_i, b_i)$.\newline
\text{  }
\\
Claim:\newline
(i) f is 1 to 1. \newline
(ii) f is onto. \newline
(iii) f preserves $\leq$. \newline
(Together, (i), (ii), (iii) show f is an isomorphism.)\newline
\text{   }
\\
(i) Since each point in $J(P)$ is unique, each path through the partition is unique, i.e. there is a 1:1 correspondence between the order ideals in $J(P)$ and the Dyck paths through the partition.  Suppose that we have two different order ideals, $I_1$ and $I_2$, that map to the same Dyck path.  Let $I_1=\{(a_1, b_1), (a_2, b_2), \ldots, (a_k, b_k)\}$ and $I_2 = \{(c_1, d_1), (c_2, d_2),\ldots,(c_l, d_l)\}$.  Since the order ideals are different, either $I_1-I_2$ is non-empty or $I_2-I_1$ is non-empty.  We may assume, without loss of generality, that it is  $I_1-I_2$ that is non-empty.  This implies that there must be at least one element $(a_i, b_i) \in I_1$ such that $(a_i, b_i)\notin I_2$.  But then, according to our procedure above, the path through the partition diagram which corresponds to $I_1$ will lie above the cell labelled $(a_i, b_i)$, while the path corresponding to $I_2$ will lie below this cell.  Since they cover different cells, $I_1$ and $I_2$ correspond to different Dyck paths.  Therefore we have contradicted the assumption that we can have two different order ideals leading to the same Dyck path.\newline
(ii)  Let $D$ be an arbitrary Dyck path such that the path lies above the cells labelled $(a_1, b_1), (a_2, b_2),\ldots,(a_k, b_k)$ in the partition diagram and where each cell corresponds to a lattice point in $P$ where $(a+b)\leq (n-2)$ and $a,b\geq 0$ with $(a,b)\leq(c,d)$ iff $a\geq c$ and $b \geq d$.  Since a Dyck path includes only north and east steps, if a Dyck path lies above the cell labelled $(a_i, b_i)$ in the partition, it must also lie above cells ($a_j, b_j)\leq(a_i, b_i)$ where $\leq$ is defined as in $P$.\newline
We construct the set of points $\{(a_1, b_1), (a_2, b_2),\ldots,(a_n, b_n)\}$ by including $(a_i, b_i)$ in the set iff the Dyck path travels above the cell with that label in the partition.  Since for every $(a_i, b_i)$, all elements less than or equal to $(a_i, b_i)$ are also included in the set, by definition this set forms an order ideal of $P(a,b)$.\newline
 (iii) Suppose that we have two order ideals, $I_1 = \{(a_1, b_1),(a_2, b_2),\ldots,(a_k, b_k)\}$ and $I_2 = \{(c_1, d_1), (c_2, d_2),\ldots,(c_l, d_l)\}$ such that $I_1 \leq I_2$.  This means that the elements of $I_1$ are a subset of the elements of $I_2$, i.e. $(a_1, b_1),(a_2, b_2),\ldots,(a_k, b_k) \in I_2$.  Then $I_1$ corresponds to a Dyck path that lies above the cells $(a_1, b_1),(a_2, b_2),\ldots,(a_k, b_k)$ and $I_2$ corresponds to Dyck path that lies above the cells $(c_1, d_1),(c_2, d_2),\ldots,(c_l, d_l)$.  But since $(a_1, b_1),(a_2, b_2),\ldots,(a_k, b_k) \in I_2$, then the Dyck path corresponding to $I_1$ must lie on or below the Dyck path corresponding to $I_2$ and therefore the $\leq$ relation is preserved.
\end{proof}

\subsection{Direct formula for M\"obius function}
\text{  }
\\
\\
\begin{definition} A poset, $P$, is a lattice if every pair of elements $(x,y)\in P$ have a greatest lower bound (or meet) and a least upper bound (or join).
\end{definition}
\begin{definition} A lattice is distributive iff the meet and join operations distribute over one another.
\end{definition}
\begin{cor} The poset of Dyck paths is a finite distributive lattice.
\end{cor}
\begin{proof}
Since the poset of Dyck paths, $D_n$, is isomorphic to $J(P)$ and posets of order ideals are always finite distributive lattices \cite{RS1}, then $D_n$ is also a finite distributive lattice.
\end{proof}
\text{  }
\\
\begin{definition}
An antichain is a subset, $A$ of a poset, $P$, such that no two distinct elements of $A$ are comparable.
\end{definition}

\noindent We shall examine the antichains of $D_n$ in detail in a later section of this paper.  However, in \cite{RS1}, Stanley notes that if $L=J(P)$ is a finite distributive lattice consisting of order ideals $I,\ldots, I'$, then M\"obius function of the interval $[I,I']$ of $L$, 
\begin{equation*} \mu(I,I')= \left\{
\begin{array}{rl}
(-1)^{\mid{I'-I}\mid}, &\text{if $I'-I$ is an antichain of $P$}\\
 0, &\text{otherwise.}
 \end{array} \right.
 \end{equation*}
\text{  }
\\
Since we have proved that there is an isomorphism between the poset of Dyck paths, $D_n$, and $J(P)$, it follows that the M\"obius function for $D_n$ is the one given above.
\text{  }
\\
\\
\begin{definition}
Stacked cells are a pair of cells, $x$ and $y$, such that cell $x$ lies immediately below or to the right of cell $y$ iff $y$ covers $x$ in the corresponding poset, $P$.
\end{definition}
\text{ }
\begin{prop}
There are stacked cells in the set $I'-I$ iff $I'-I$ is not an antichain of $P$.
\end{prop}
\begin{proof}
\noindent Let $I$ and $I'$ be order ideals in $J(P)$ which correspond to elements in the Dyck path poset, $D$. We consider $I'-I$ to be the difference set consisting of only elements of $I'$ that are not also elements of $I$.\newline
($\Rightarrow$) Assume that there are stacked cells in the set $I'-I$.  Then there is at least one element, $y \in (I'-I)$ that covers an element $x \in (I'-I)$.  But a cover relationship can only exist between comparable elements.  Since $x$ and $y$ are comparable, and both $x,y \in (I'-I)$, then by the definition of antichain, $I'-I$ is not an antichain of $P$.\newline
($\Leftarrow$) Now, assume that $I'-I$ is not an antichain of $P$. Then there must be at least two different elements of $I'-I$ which are comparable, that is some $x, y \in (I'-I)$ such that $x < y$.  Elements which are comparable must belong to a chain.  Since $x < y$, then there is a chain with length $\geq 1$ such that $x=x_0<x_1<\ldots <x_l=y$ and such that each $x_i \in I'-I$. But, since there is at least one $x_1$ that covers $x$, cell $x_1$ must lie immediately above or to the left of cell $x$.  By the definition, we conclude that there are stacked cells in the set $I'-I$.\newline
\end{proof} 

\newpage
\section{Counting the number of chains in $D_n$}
\subsection{Total number of chains}
\text{  }
\\
\\
Next we look at some interesting subsets within $D_n$, beginning with the chains.
\begin{definition}
A chain, $C$, is a totally ordered subset of a poset, $P$.  Within a chain any two elements are comparable.
\end{definition}
\begin{definition}
The length of a finite chain, $l(C)$, is given by the equation $l(C)=\vert C \vert -1$, where $\vert C \vert$ represents the number of elements in $C$.
\end{definition}
\begin{example}
In the $D_3$ poset, there are 24 chains.  This total is made up of 1 chain on 0 vertices, 5 chains on 1 vertex, 9 chains on 2 vertices, 7 chains on 3 vertices and 2 chains on 4 vertices.  As an example, here is the $D_3$ poset and the 7 chains on 3 vertices:
\begin{figure}[htbp]
	\centering
		\includegraphics[width=5in]{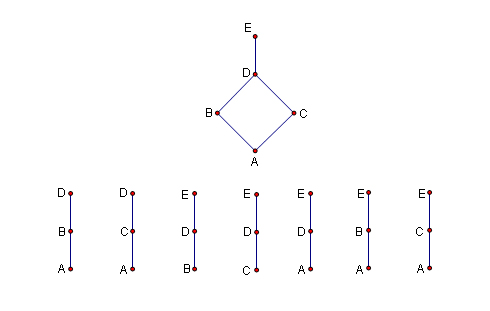}
	\caption{Chains on 3 vertices in $D_3$}
	\label{Chains on 3 vertices in $D_3$}
\end{figure}
\end{example}
\text{ }
\\
\noindent While it is possible to count the chains by hand for $D_n$ posets where $n$ is small, this strategy will prove challenging for large $n$ values.  However, as outlined in \cite{MB}  and \cite{RS1}, we can use the zeta function and some basic matrix algebra to find the total number of chains, $x=x_0<x_1<\ldots<x_k = y$, from $x$ to $y$ in the poset. Recall that the zeta function is defined as $\zeta(x,y)=1$ for all $x \leq y$ in $D_n$ and 0 otherwise. 
\text{ }
\\
\begin{definition} The identity function, denoted $\delta$, is defined such that $\delta (x,y) = 1$ if $x = y$ and $0$ otherwise.
\end {definition}
Note that either the $\zeta (x,y)$ or $\delta (x,y)$ functions may alternatively be considered as matrices whose rows and columns are indexed by the elements of $D_n$.  The proofs below will make use of both conceptualizations.\newline
\text{ } 

\begin{theorem}{(from \cite{RS1}, p. 115)}\newline
Let $x=x_0<x_1<\ldots<x_k = y$ be a chain on the interval $[x,y]$ in a poset, $P$, of $n$ elements. Then the total number of chains of length $k$ on the interval $[x, y]$ is equal to $(\zeta - \delta)^k (x, y)$. 
\end{theorem}
\begin{proof}
Let $T(k)$ be the statement: The total number of chains of length $k$ on the interval $[x, y]$ is equal to $(\zeta - \delta)^k (x, y)$.
We will show that $T(k)$ is true for all non-negative integers, $k$, using induction.\newline
i) Base cases:  
Set $k = 0$. It is clear in the poset that chains of length 0 occur when $[x,y]$=$[x,x]$, i.e. when $x=y$ in $P$.  But $(\zeta - \delta)^0 (x,y) = \delta(x,y)$.  Therefore $T(k)$ is true for $k=0$ and the total number of chains of length 0 on the interval $[x,y]$ = the number of $[x,x]$ intervals in the poset = the number of Dyck paths in $D_n$, if this is the poset, $P$, under consideration.
\newline
Now, set $k=1$.  Chains of length 1 are of the form $x=x_1<x_2=y$ and hence the number of such chains is equal to 1 if $x<y$ and 0 if $x \nless y$ and this is equal to $(\zeta - \delta) (x, y)$ and thus $T(k)$ is true for $k=1$.\newline
ii) Assume that $T(k-1)$ is true.  So we are assuming that the total number of chains of length $k-1$ on the interval $[x,y]$ is equal to $(\zeta - \delta)^{k-1} (x, y)$.\newline
iii) We show that this implies $T(k)$ is true for chains of length $k$ in the poset of $n$ elements. Any chain of length $k$ can be decomposed into a chain of length $k-1$, $x=x_0<x_1<\ldots<x_{k-1}=z$,  and a chain consisting of 2 elements, $z=x_{k-1}<x_k = y$.  Fix a value $z$ where $z \in [x,y]$.  Now consider the product:
 
$$(\zeta - \delta)^{k-1} (x,z) \cdot (\zeta - \delta) (z,y).$$

By the multiplication principle, the expression $(\zeta - \delta)^{k-1}(x,z) \cdot  (\zeta - \delta)(z,y)$ is equal to the number of pairs, $(a,b)$, where $a$ belongs to a set, $A$, with cardinality $(\zeta - \delta)^{k-1} (x,z)$, and where $b$ belongs to a set, $B$, with cardinality $(\zeta - \delta) (z,y)$.  In our case, $A$ is the set of chains of length $k-1$ which start at $x$ and end at $z$, and this set is counted by the entry in the $n \times n$ matrix, $(\zeta - \delta)^{k-1} (x,z)$.  Similarly, for our purposes $B$ is the set of chains from $z$ to $y$ for this choice of $z$, and this set is counted by the entry in the $n \times n$ matrix $(\zeta - \delta) (z,y)$.  Note also that since this second number is counting chains of length 1, its value will equal 1 if $z<y$ and 0 if $z \nless y$.\newline

Next, we use the same argument, but consider other choices of $z$ on $[x,y]$.  Since these choices represent disjoint cases, we apply the addition principle to sum over all possible choices of $z$.  So the expression
 
$$\sum_{z\in [x,y]} (\zeta - \delta)^{k-1} (x,z) \cdot (\zeta - \delta) (z,y)$$                    
 
represents the number of pairs whose first element is a chain of length $k-1$ on $[x,z]$ and whose second element is a chain of length 1 on $[z,y]$, for all $z \in [x,y]$.\newline  

%
 
Since $(\zeta - \delta)^{k-1}$ and $(\zeta - \delta)$ are both $n \times n$ matrices, we can multiply using ordinary matrix multiplication.  Therefore
 
$$\sum_{z \in P} (\zeta - \delta)^{k-1} (x,z) \cdot (\zeta - \delta) (z,y) = (\zeta - \delta)^{k} (x,y) = \text{the number of chains on $[x,y]$ in $P$}$$
 
Since $T(k-1)$ true implies $T(k)$ true, by the principle of mathematical induction we have shown that $T(k)$ is true for all non-negative integers, $k$.

\end{proof}

\begin{theorem}\label{chainsproof}{(from \cite{RS1}, p. 115)}\newline
The total number of chains $x=x_0<x_1<\ldots<x_k = y$ from $x$ to $y$ in $D_n$ is equal to $(2\delta -\zeta)^{-1}(x,y)$.
\end{theorem}
\begin{proof} 
Let $x=x_0, x_1, x_2, \ldots, x_k=y$ be Dyck paths in $D_n$, with the $\zeta$ and $\delta$ functions (or matrices) defined as above.\newline
Therefore $(\zeta -\delta)(x,y)$ = $1$ if $x < y$ and $0$ otherwise (including if $x = y$).  From the previous theorem, $(\zeta -\delta)^k$ counts the number of chains of length $k$ from $x$ to $y$.\newline
Let $l$ be the length of the longest chain in the interval $[x,y]$.  Then $(\zeta -\delta)^{l+1}(x_i, x_j)=0$ for all $x \leq x_i \leq x_j \leq y$.\newline
Note that the difference matrix $(2\delta -\zeta)$ is upper triangular with ones along the main diagonal.  The determinant ($=1$) is clearly non-zero, and therefore this matrix is invertible.\newline
Now, consider the product
\begin{align*}
&  (2\delta -\zeta)[1 +(\zeta - \delta) + (\zeta - \delta)^2 + \ldots + (\zeta - \delta)^l](x_i, x_j)\\
&= (\delta - (\zeta - \delta))[1 +(\zeta - \delta) + (\zeta - \delta)^2 + \ldots + (\zeta - \delta)^l](x_i, x_j)\\
&= [\delta - (\zeta - \delta)^{l+1}](x_i, x_j)\\
&= \delta(x_i, x_j)
\end{align*}
\text {  }
Because the product is equal to the identity, it follows that 
$$(2\delta -\zeta)^{-1}=1 + (\zeta - \delta) + (\zeta - \delta)^2 + \ldots + (\zeta - \delta)^l$$ on the interval $[x, y]$.  But since $l$ was defined to be the longest chain over the interval, then the total number of chains from $x$ to $y$ is given by $1 + (\zeta - \delta) + (\zeta - \delta)^2 + \ldots + (\zeta - \delta)^l$ where for each $k$, $(\zeta - \delta)^k$ counts chains of length $k$.
\end{proof}

Using Maple software, we input the zeta matrix and then have the software compute the $(2\delta -\zeta)^{-1}$ matrix.  Summing all of the entries in this matrix and adding one to the total (for the empty chain), we obtain the total number of chains for the poset.
The total number of chains in $D_n$ for $n=0$ to $5$ are listed in the table below. We have added this sequence to \cite{OLEIS} as A143672.\newline
\text{  }
\\
\begin{center}
\begin{tabular}[h]{l c r}

\scriptsize
n &  \scriptsize Number of Chains \\
\hline

\scriptsize 0 & \scriptsize 1\\
\scriptsize 1 & \scriptsize 2\\
\scriptsize 2 & \scriptsize 4\\
\scriptsize 3 & \scriptsize 24\\
\scriptsize 4 & \scriptsize 816\\
\scriptsize 5 & \scriptsize 239968\\
\end{tabular}
\end{center}
\text{  }

\subsection{Chain polynomial}
\begin{definition}
The chain polynomial, $C(P, t) = 1 + \sum_{k} c_k t^{k+1}$ where $c_k$ is the number of chains (or totally ordered subsets) in a poset $P$ of length $k$.  The exponent $k+1$ indicates the number of vertices in the chain and the 1 in front of the sum denotes the empty chain.
\end{definition}

We also use Maple to find the chain polynomial for the $D_n$ poset.  Since $(\zeta - \delta)^k$ counts chains of length $k$, the chain polynomial for $D_n$ is easily determined by choosing values $k=0$ to $l$ (where $l$ is the length of the longest chain in the poset) for each value of $n$.  The results are summarized in the table:\newline
\text{   }
\\
\begin{tabular}[h]{l c c}
\scriptsize
n &  \scriptsize Chain Polynomial & \scriptsize Factored Form \\
\hline

\scriptsize 0 & \scriptsize 1 \\
\scriptsize 1 & \scriptsize $1+t$ & \scriptsize $(1+t)$\\
\scriptsize 2 & \scriptsize $1 + 2t + t^2$ & \scriptsize $(1+t)^2$\\
\scriptsize 3 & \scriptsize $1 + 5t + 9t^2 + 7t^3 + 2t^4$ & \scriptsize $(1+2t)(1+t)^3$\\
\scriptsize 4 & \scriptsize $1+14t+70t^2+176t^3+249t^4+202t^5+88t^6+16t^7$ & \scriptsize $(1+8t+8t^2)(1+2t)(1+t)^4$\\
\scriptsize 5 & \scriptsize $1+42t+552t^2+3573t^3+13609t^4+33260t^5+54430t^6+$\\
\scriptsize   & \scriptsize $60517t^7+45248t^8+21824t^9+6144t^{10}+768t^{11}$ & \scriptsize $(1+37t+357t^2+1408t^3+2624t^4+2304t^5+768t^6)(1+t)^5$
\end{tabular}
\text{}
\\
\\
In this table we can see that for each $n$, the coefficient of $t$ is equal to $C_n$.  This makes sense since the number of chains on the interval $[x,x]$ in the poset will be equal to the number of vertices in the Hasse diagram which in turn corresponds to the number of Dyck paths in $D_n$ (which are enumerated by the Catalan numbers).  
\text{ }
\\
\subsection{Maximal and maximum chains}
\text{  }
\\
\\
\begin{definition}
Let $P$ be a poset.  A chain, $C \in P$ is a maximum chain if no other chain of $P$ has greater length.  A chain, $C \in P$ is a maximal chain if it cannot be extended, i.e. the addition of any other element to the chain would result in it containing at least one pair of incomparable elements.
\end{definition}
\begin{definition}
A finite poset, $P$, is ranked (or graded) if all of its maximal chains have the same length, $n$.  Every ranked poset has a unique rank function, $\rho :P \rightarrow \{0, 1, \ldots, n\}$, where $\rho(x)=0$ if $x$ is a minimal element in $P$ (that is, there is no $x' \in P$ such that $x > x'$).  If $x$ is not a minimal element in $P$, and $y$ covers $x$ in $P$, then the rank function, $\rho (y)=\rho (x) + 1$.  An element $x \in P$ has rank $i$ if $\rho(x)=i$.
\end{definition}
\begin{prop} $D_n$ is a graded poset.
\end{prop}
It follows from these definitions that the length (or rank) of any Dyck path poset will be equal to the length of its maximum chain(s).   Thus, $D_n$ is a graded poset where the rank function is the area between the Dyck path and the line $y=x$.  Furthermore, one Dyck path, $D$, covers another path, $D'$, if $D' = D$ minus a cell.\newline  

Moreover, since no two elements having the same rank (which represent Dyck paths covering the same area) in $D_n$ are comparable, the maximum chains of $D_n$ are identical to the maximal chains of $D_n$. Therefore, throughout the remainder of this paper, note that any references to numbers of maximal chains also indicate numbers of maximum chains.  The number of maximal chains can be calculated with the formula $(\zeta -\delta)^l$ where $l = $[(number of vertices in the maximal chain) $-1] = $(number of edges in the maximal chain)$= \binom{n}{2}$ is the length of the maximal chain.
In the table of chain polynomials above, the number of maximal chains is equal to the coefficient in front of the largest power of $t$.  This sequence is included in \cite{OLEIS} as A005118.  The table below lists the number and length of maximal chains for $n=0$ to $5$.
\text{ }
\\
\begin{center}
\begin{tabular}[h]{l c c}
\scriptsize
n &  \scriptsize Number of Maximal Chains &  \scriptsize Length of Maximal Chain\\
\hline

\scriptsize 0 & \scriptsize 1 & \scriptsize 0\\
\scriptsize 1 & \scriptsize 1 & \scriptsize 0\\
\scriptsize 2 & \scriptsize 1 & \scriptsize 1\\
\scriptsize 3 & \scriptsize 2 & \scriptsize 3\\
\scriptsize 4 & \scriptsize 16 & \scriptsize 6\\
\scriptsize 5 & \scriptsize 768 & \scriptsize 10\\
\end{tabular}
\end{center}
\text{ }
\\

Stanley in \cite{RS1} offers a different version of this formula, which can be proved using an argument similar to that in theorem \ref{chainsproof} above:  the number of maximal chains on $[x,y]$ is equal to $(\delta -\eta)^{-1}(x,y)$ where $\eta(x,y)=1$ if $y$ covers $x$ and $0$ otherwise.

\subsection{Bijection with Young tableaux}
\text{ }
\\
\\
Alternatively, we demonstrate a second method for counting the number of maximal chains in $D_n$.\newline
\begin{prop}
The number of maximal chains in $D_n$ is equal to the number of standard Young tableaux for the staircase partition of size $\binom{n}{2}$ above the minimal Dyck path in $D_n$ and therefore are counted using the formula
\begin{equation}
\frac{\binom{n}{2}!}{\prod_{i=1}^{n-1}(2i-1)^{n-i}}.
\end{equation}
\label{tableau}
\end{prop}

However, before proving this claim, we need to define some terms and establish the correspondence between Dyck paths and Young diagrams.
\begin{definition} A partition, $\lambda$, is a sequence of positive integers $\lambda_1, \lambda_2, \ldots, \lambda_k$ of weakly decreasing size.  The $\lambda_i$ are called the parts of the partition and the sum of the parts, $\lambda_1 + \lambda_2 + \ldots + \lambda_k$, gives the area of $\lambda$.
\end{definition}
Each partition can be represented graphically by a Young (Ferrer) diagram.
\begin{definition} A Young diagram is a set of cells arranged in rows of weakly decreasing size which are aligned on the left.  Each row of cells in the diagram corresponds to a part, $\lambda_i$, in the partition.  The Young diagrams in this paper will follow the English notation so that the row of largest size is on the top of the diagram. Therefore we will refer to the top row of the diagram (which corresponds to $\lambda_1$ in the partition) as row 1 and then number the rows below using consecutive positive integers. 
\end{definition}
\begin{definition} A standard Young tableau is a Young diagram in which the cells are filled with the positive integers, $1, 2, \ldots, n$, in such a way that the values along each row (from left to right) and column (from top to bottom) are increasing and in a tableau of $n$ cells, each integer from $1$ to $n$ will occur exactly once.
\end{definition}

First we note the obvious bijection between Dyck paths and the Young diagrams that occupy the grid cells lying above the paths.  If we place the Young diagram onto an $n \times n$ grid such that the leftmost cell in row one of the diagram occupies the leftmost cell in the top row of the grid, then the Dyck path will follow the exposed bottom and righthand edges of the partition. In addition, it must have a vertical segment joining the path to $(0,0)$ and a horizontal segment joining the path to $(n,n)$.  Conversely, given a Dyck path, the Young diagram can be constructed by starting at the left side of the first horizontal edge in the Dyck path, tracing along the path until reaching the upper end of the last vertical segment and then joining the end points of this subpath using lines drawn along the left and top borders of the grid.\newline
Now, if $D_n$ is a Dyck path poset of order $n$ and $d \in D_n$ is the path with an area of 0 (the staircase path that contains no complete lattice cells between it and the main diagonal), then the number of complete lattice cells above this path will be equal to $\frac{n(n-1)}{2}$ since it forms a staircase partition with rows of sizes $1, 2, \ldots, n-1$ from bottom to top.  Since all Dyck paths are formed using only N and E steps and travel from $(0,0)$ to $(n,n)$, every path in $D_n$ can be uniquely mapped to a valid partition diagram (lying above the Dyck path within the grid) that contains at most $n-1$ rows. Recall that in order for a partition to be considered `valid', the size of any row $i$ must be $\leq$ the size of any rows above $i$. We introduce a vector notation to record the parts of the partition occupying each row in the grid. If we number the rows of the $n \times n$ grid from top to bottom with $1, 2, \ldots, n$, then we can represent these partition diagrams as ordered $(n-1)$-tuples, $(x_1, x_2, \ldots, x_{n-1})$, in which $x_i$ is the size of the $i$th row in the partition diagram.  It is evident that the first $k$ entries in this vector will correspond to $\lambda_1, \lambda_2, \ldots, \lambda_k$ in the partition and that in order to lie above the Dyck path, the maximum size for the $i$th row is $n-i$.\newline
\begin{example}   This partition of $D_5$ would be encoded by the tuple $(4,2,1,0)$.  The corresponding Dyck path from $(0,0)$ to $(5,5)$ is also shown.
\begin{figure}[htbp]
	\centering
		\includegraphics[width=2.5in]{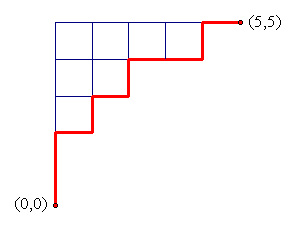}
	\caption{The $(4, 2, 1, 0)$ partition of $D_5$ and its corresponding Dyck path}
	\label{The $(4, 2, 1, 0)$ partition of $D_5$}
\end{figure}
\end{example}
\text{ }
\\
\\
We are now ready to prove proposition \ref{tableau}.

We do this by establishing a bijection between the number of maximal chains in $D_n$ and the number of standard Young tableaux in the $\frac{n(n-1)}{2}$ staircase partition above the minimal Dyck path in $D_n$.

\begin{proof}
Let $D_n={d_1, d_2, \dots, d_k}$ be the Dyck path poset.  There are $k$ distinct Young diagrams within the grid cells above the Dyck path, in 1:1 correspondence with ${d_1, d_2, \dots, d_k}$.  Each diagram can be represented by a unique $(n-1)$-tuple, $(x_1, x_2, \ldots, x_{n-1})$, in which $x_i$ is the size of the $i$th row in the Young diagram.  In order for one Dyck path to cover another in the poset, it is necessary that the partition diagrams differ by exactly one cell, and therefore the corresponding tuples differ by exactly one for one of the $x_i$.\newline
A maximal chain in $D_n$ will include one partition of each size from $0$ to $\binom{n}{2}$ inclusive since these correspond to the ranks in the poset.  Recall that no two Dyck paths of the same area are comparable, so only one from each rank is permitted in the maximal chain.  A maximal chain in the poset is constructed starting with the empty partition with $(n-1)$-tuple $(0, 0, \ldots, 0)$ and increasing one of the  $x_i$ by one at each stage, such that a valid partition is created, until the staircase partition with $(n-1)$-tuple $(n-1, n-2, \ldots, 1)$ is obtained.  In terms of the Young diagrams, this means that once the first cell has been placed, the partition is increased at each stage by adding one cell to the right of or below an existing cell. This is equivalent to changing the $(n-1)$-tuple consisting of all zeros to the $(n-1)$-tuple $(n-1, n-2, \ldots, 1)$ where at all stages $(x_1 \geq x_2 \geq \ldots \geq x_{n-1})$. The transformation occurs by increasing one of the entries by one at each step such that the maximum value for any entry is not exceeded.  The number of ways of carrying out this procedure will be equal to the number of maximal chains for the poset.
But, now imagine that at each stage in the construction, as a cell is added to the partition it is assigned the next integer in the sequence $1, 2, \dots, n-1$.  With this numbering process, we have created a standard Young tableau.  So, the number of different ways of building the partition diagram (up to the staircase shape) with the cells labelled  as indicated will equal the number of valid ways of changing the $(n-1)$-tuple consisting of all zeros to the $(n-1)$-tuple $(n-1, n-2, \ldots, 1)$, and will equal the number of maximal chains for the poset.
\end{proof}

Conveniently, the number of standard Young tableaux for a partition is straightforward to calculate using the hook length formula \cite{GA}, \cite{VLW}.
\begin{theorem}{(Robinson-Frame-Thrall Theorem)}\newline
The hook length of a cell, $x$, which lies in a Young diagram for a partition, $\lambda$, can be calculated by summing the number of cells in $\lambda$ that lie to the right of $x$ and the number of cells in $\lambda$ that lie below $x$ and then adding one for the cell itself.  If a partition, $\lambda$, consists of $n$ cells, then: $$\text{the number of standard Young tableaux in $\lambda$} = \frac{n!}{\prod_{x \in \lambda}hook(x)}.$$
\end{theorem}
A proof of this theorem is provided in \cite{VLW}.\newline
\text{  }
\\
Since we are dealing only with staircase partitions, we can specialize this formula.  For $D_n$, staircase partitions above the minimal Dyck path always consist of $\binom{n}{2}$ cells, so we use $\binom{n}{2}!$ for the numerator.  The hook lengths for our diagrams always consist of $(n-1)$ ones, $(n-2)$ threes, $(n-3)$ fives, and so on.  Therefore the product of the hook lengths in the denominator will be $\prod_{i=1}^{n-1}(2i-1)^{n-i}$.  We use the specialized formula in the solution for the example below.
\begin{example}
In the diagram below (the staircase partition for $D_5$), the cells have been labelled with their hook lengths.\newline

\begin{figure}[h]
	\centering
		\includegraphics[width=2in]{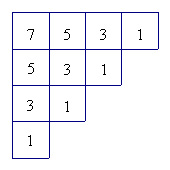}
	\caption{Hook lengths for $D_5$ staircase partition}
	\label{Hook lengths}
\end{figure}
\text{  }
\\
Then, by the hook length formula, the number of standard Young tableaux corresponding to this diagram is equal to  \begin{align*}
\frac{\binom{n}{2}!}{\prod_{i=1}^{n-1}(2i-1)^{n-i}}&= \frac{\binom{5}{2}!}{\prod_{i=1}^{4}(2i-1)^{5-i}}\\
 &= \frac {10!}{(1^4)\cdot(3^3)\cdot(5^2)\cdot(7^1)}\\
 &= \frac {3628800}{4725}\\
 &= 768
 \end{align*}
 \text{  }
which is equal to the number of maximal chains in $D_5$.
\end{example}

\newpage
\section{Counting the number of antichains in $D_n$}
\subsection{Total number of antichains}
The next subsets of $D_n$ that we consider are the antichains.
\begin{definition}
Let $P$ be a poset.  An antichain is a subset, $A \in P$, in which no two distinct elements are comparable.
\end{definition}
\begin{example}
In the $D_3$ poset, there are 7 antichains:  \{\text{ }\}, \{A\}, \{B\}, \{C\}, \{D\}, \{E\}, \{B,C\}.
\begin{figure}[htbp]
	\centering
		\includegraphics[width=1.2in]{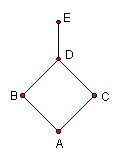}
	\caption{$D_3$ poset}
	\label{$D_3$ poset}
\end{figure}
\end{example}
We are interested in counting the number of antichains in $D_n$ posets for other values of $n$.  Interestingly, the values obtained match those found when the order ideals of $D_n$ are counted.  Recall that an order ideal (or downset) of a poset, $P$, is a subset, $I \in P$, such that if $x \in I$ and $y \leq x$, then $y \in I$.
In figure \ref{$D_3$ poset}, there are seven order ideals in the $D_3$ poset: \{\text{ }\}, \{A\}, \{BA\}, \{CA\}, \{CBA\}, \{DCBA\}, \{EDCBA\}. \newline

\begin{prop} {(from \cite{RS1}, p. 100)}\newline
Let $A=\{A_1, A_2, \ldots, A_k\}$ be the set of antichains in $D_n$ and let $I=\{I_1, I_2, \ldots, I_l\}$ be the set of order ideals in $D_n$.  Then we claim that $k=l$ (the number of elements in $A$ is equal to the number of elements in $I$).
\end{prop}
\begin{proof}
We will prove this claim by establishing a bijection between the sets $A$ and $I$.\newline
We define $f: A \rightarrow I$ by taking $A_j \in A$ to be the maximal element of $I_m \in I$ where by definition of an order ideal,  $I_m$ consists of $A_j$ plus any elements in $D_n$ which are less than or equal to $A_j$.  We will show that the function $f$ is both 1:1 and onto.\newline
\text{  }
\\
i)  1:1\newline
Suppose that we have two different antichains, $A_x, A_y \in A$.  Then, by our function definition, $A_x$ will correspond to the order ideal $I_x=\{A_x, d_1, d_2, \ldots, d_r\}$ where each of the $d_i \in D_n$ are $\leq A_x$ and $A_y$ will be mapped to the order ideal $I_y=\{A_y, d_1, d_2, \ldots, d_s\}$ where each of the $d_j \in D_n$ are $\leq A_y$.  Clearly if $A_x = A_y$ then we would have the same order ideal.  Since we started with two different antichains, there must be at least one element in $A_x$ that is not in $A_y$ or vice versa.  Without loss of generality, we may assume that this extra element is in $A_x$.   Then $I_x$ must be different than $I_y$ since they have different sets  of maximal elements and all remaining elements (Dyck paths) in both order ideals are less than or equal to the set of  maximal elements.\newline
\text{  }
\\
ii) onto\newline
Now, suppose that $I_x$ is an arbitrary order ideal $\in I$.  Then it will consist of a set of maximal elements, $i_m$,  and a set of Dyck paths that are $\leq i_m$.  But since $I_x$ is an order ideal in $D_n$, the set of elements, $i_m$,  must consist of one or more Dyck paths and since $i_m$ is maximal, then it is either a single Dyck path or a set of incomparable Dyck paths.  In either case, by definition of antichain, the set of maximal elements, $i_m$, corresponds to a unique antichain in $D_n$ which under the definition of $f$ will map onto the order ideal $I_x$.
\end{proof}
\text{ }
\\
As we saw in a previous section, sets of order ideals will themselves form posets if they are ordered by inclusion.  In fact, we used a specially defined set of order ideals to determine the M\"obius function for $D_n$.  But, now we continue the discussion about antichains by listing the number of antichains in $D_n$ for $n=0$ to $5$.  The center column of this chart details the breakdown according to the number of elements in (or size of) the antichain beginning with antichains of size 0 and increasing.  These data were obtained using MuPAD-Combinat \cite{webMuPAD}.  \newline
\begin{center}
\begin{tabular}[h]{l l c}
\scriptsize n & \qquad\scriptsize Number of Antichains in $D_n$ &  \qquad\scriptsize Total\\
\scriptsize {  } & \qquad\scriptsize (arranged by increasing size) & \qquad\scriptsize {  }\\
\hline

\scriptsize 0 & \qquad\scriptsize 1, 1 & \qquad\scriptsize 2\\
\scriptsize 1 & \qquad\scriptsize 1, 1 & \qquad\scriptsize 2\\
\scriptsize 2 & \qquad\scriptsize 1, 2 & \qquad\scriptsize 3\\
\scriptsize 3 & \qquad\scriptsize 1, 5, 1 & \qquad\scriptsize 7\\
\scriptsize 4 & \qquad\scriptsize 1, 14, 21, 6 & \qquad\scriptsize 42\\
\scriptsize 5 & \qquad\scriptsize 1, 42, 309, 793, 810, 348, 56, 2 & \qquad\scriptsize 2361\\
\end{tabular}
\end{center}
We have added this sequence to \cite{OLEIS} as A143673.
\text{ }
\\
\subsection{Maximal antichains}
We define maximal and maximum antichains in an analogous way to the maximal and maximum chain definitions from the previous section.\newline
\begin{definition}  An antichain in a poset, $P$, is a maximal antichain if it cannot be extended, meaning that the addition of any other element of $P$ to the antichain would result in it containing at least one pair of comparable elements and therefore violating the antichain condition.
\end{definition}
This definition implies that no maximal antichain is a proper subset of any other antichain.  We will use this property to enumerate the total number of maximal antichains for $D_n$.
However, we will first establish the following claim.
\begin{prop}
The elements in each row on the Hasse diagram form a maximal antichain for $D_n$.
\end{prop}
\begin{proof}
Let $D_n$ be the poset of Dyck paths of order $n$.  Then $D_n$ can be represented by a Hasse diagram such that each row of the diagram corresponds to a rank, and that the elements in the poset are ranked according to the area they cover.  For $D_n$, no two elements having the same rank are comparable, so each row forms an antichain.  The addition of any other Dyck path to the antichain formed by the elements of a given rank will result in at least one pair of comparable elements and thus the elements of each row of the Hasse diagram form a maximal antichain in $D_n$.  
\end{proof}
Note that the empty set, while an antichain, does not satisfy the criteria to be a maximal antichain in $D_n$ for $n \geq 1$ since it is possible to add another element from $D_n$ to $\{\text{ } \}$ and still have an antichain.  We can calculate the number of rows in the Hasse diagram for $D_n$ using the formula $\binom{n}{2} + 1$ since it will include ranks corresponding to areas ranging from size $0$ to $\binom{n}{2}$.  Thus, the number of maximal antichains for $D_n$ must be greater than or equal to $\binom{n}{2} + 1$.\newline
  
The table below lists the number of maximal antichains in $D_n$ for $n=0$ to $5$.  The center column of this chart details the breakdown according to the number of elements in (or size of) the antichain beginning with antichains of size 0 and increasing.  The data were obtained using MuPAD-Combinat \cite{webMuPAD}.

\begin{center}
\begin{tabular}[h]{l l c}
\scriptsize n & \qquad\scriptsize Number of Maximal Antichains in $D_n$ &  \qquad\scriptsize Total\\
\scriptsize {  } & \qquad\scriptsize (arranged by increasing size) & \qquad\scriptsize {  }\\
\hline

\scriptsize 0 & \qquad\scriptsize 1 & \qquad\scriptsize 1\\
\scriptsize 1 & \qquad\scriptsize 0, 1 & \qquad\scriptsize 1\\
\scriptsize 2 & \qquad\scriptsize 0, 2 & \qquad\scriptsize 2\\
\scriptsize 3 & \qquad\scriptsize 0, 3, 1 & \qquad\scriptsize 4\\
\scriptsize 4 & \qquad\scriptsize 0, 3, 8, 6 & \qquad\scriptsize 17\\
\scriptsize 5 & \qquad\scriptsize 0, 3, 14, 62, 132, 124, 42, 2 & \qquad\scriptsize 379\\
\end{tabular}
\end{center}

\text{ }
\\
We have added this sequence to \cite{OLEIS} as A143674.
\text{ }
\\
\subsection{Maximum antichains}
\begin{definition}
An antichain in a poset, $P$, is a maximum antichain if there are no other antichains of greater size in $P$.  The cardinality or size of the maximum antichain(s) is the width of the poset.
\end{definition}
\text{ }
\\
Therefore in order to determine the size of the maximum antichains we need to know the breakdown according to rank for the $C_n$ elements in $D_n$.  These are listed in \cite{OLEIS} as A129176.    The recursion used to calculate these values is: $C_0(q)=1; C_{n+1}(q)=\sum_{k=0}^n q^{(k+1)(n-k)}\cdot C_k \cdot C_{n-k}$ which gives a polynomial in $q$.  When the coefficients are listed in order of increasing powers of $q$, we obtain the number of Dyck paths in $D_n$ ordered by decreasing rank.  We are able to derive this equation using $C_n^{inv}(q)$ and the $C_n^{area}(q)$ recursion.
Recall that
\begin{align}
C_n^{inv}(q) &= \sum_{D \in D_n} q^{inv(D)}\nonumber\\
&= \sum_{D \in D_n} q^{\binom{n}{2}-area(D)}\nonumber\\
&= q^{\binom{n}{2}}\sum_{D \in D_n}(1/q)^{area(D)}\nonumber\\
&= q^{\binom{n}{2}}C_n^{area}(1/q)\label{cninvarea}\\
\nonumber \end{align}
and
\begin{equation} 
C_{n+1}^{area}(q)=\sum_{k=0}^{n}q^kC_k^{area}(q)C_{n-k}^{area}(q) \label{cnarea}
\end{equation}
\newline
Now, using \eqref{cninvarea} and \eqref{cnarea} we derive the recurrence relation:
\begin{align*}
C_{n+1}^{inv}(q) &= q^{\binom{n+1}{2}}C_{n+1}^{area}(1/q)\\
&= q^{\binom{n+1}{2}}\sum_{k=0}^{n}q^{-k}C_k^{area}(1/q)C_{n-k}^{area}(1/q)\\
&= q^{\binom{n+1}{2}}\sum_{k=0}^{n}q^{-k}q^{-\binom{k}{2}}C_k^{inv}(q)q^{-\binom{n-k}{2}}C_{n-k}^{inv}(q)\\
&= \sum_{k=0}^n q^{\frac{2n-2k^2-2k+2nk}{2}}C_k^{inv}(q)C_{n-k}^{inv}(q)\\
&= \sum_{k=0}^n q^{(k+1)(n-k)}C_k^{inv}(q)C_{n-k}^{inv}(q)\\
\end{align*}


Here the first few elements of the sequence generated by this recursion are arranged in a table where the $n$th row has $1+(n(n-1))/2$ terms (for $n \geq 1$) and each row sum is the Catalan number, $C_n$.\newline
\begin{center}
\begin{tabular}[h]{l l c}
\scriptsize
n &  \qquad\scriptsize Number of Dyck paths in $D_n$ by decreasing rank & \scriptsize Total number of paths\\
{ } & \qquad\scriptsize { } & \scriptsize $(C_n)$\\
\hline
\scriptsize 0 & \qquad\scriptsize 1 & \scriptsize 1\\
\scriptsize 1 & \qquad\scriptsize 1 & \scriptsize 1\\
\scriptsize 2 & \qquad\scriptsize 1,1 & \scriptsize 2\\
\scriptsize 3 & \qquad\scriptsize 1,1,2,1 & \scriptsize 5\\
\scriptsize 4 & \qquad\scriptsize 1,1,2,3,3,3,1 & \scriptsize 14\\
\scriptsize 5 & \qquad\scriptsize 1,1,2,3,5,5,7,7,6,4,1 & \scriptsize 42\\
\scriptsize 6 & \qquad\scriptsize 1,1,2,3,5,7,9,11,14,16,16,17,14,10,5,1 & \scriptsize 132\\
\scriptsize 7 & \qquad\scriptsize 1,1,2,3,5,7,11,13,18,22,28,32,37,40,44,43,40,35,25,15,6,1 & \scriptsize 429\\
\end{tabular}
\end{center}
From this table, it is easy to determine the size of the maximum antichains in $D_n$ by finding the largest value within each set of coefficients.  However, to find the number of maximum antichains it is not sufficient to simply count the frequency of occurrence of the largest value in each set in the table above since some maximum antichains may not consist entirely of elements having the same rank. Instead, we refer back to the MuPAD data from the table listing the total number of antichains as it also detailed the breakdown according to size. The size and number of maximum antichains in $D_n$ are listed in the chart below:\newline
\begin{center}
\begin{tabular}[h]{l c c}
\scriptsize
n &  \qquad\scriptsize Size of Maximum Antichain(s) & \scriptsize Number of Maximum Antichains in $D_n$\\
\hline
\scriptsize 0 & \qquad\scriptsize 1 & \scriptsize 1\\
\scriptsize 1 & \qquad\scriptsize 1 & \scriptsize 1\\
\scriptsize 2 & \qquad\scriptsize 1 & \scriptsize 2\\
\scriptsize 3 & \qquad\scriptsize 2 & \scriptsize 1\\
\scriptsize 4 & \qquad\scriptsize 3 & \scriptsize 6\\
\scriptsize 5 & \qquad\scriptsize 7 & \scriptsize 2\\
\end{tabular}
\end{center}
\text{ }
\\
Using Dilworth's Theorem (which applies since $D_n$ is a finite poset) the numbers in the centre column of this table also tell us the size of the smallest chain cover for the poset \cite{MB},\cite{VLW}.  
\begin{theorem}{(Dilworth's Theorem)}
For any finite poset, $P$, the size of any maximum antichain is equal to the minimum number of disjoint chains needed to fully partition $P$.
\end{theorem}
\begin{proof}
(based on \cite{MB})\newline
i)	 Let $P$ be a finite poset and $A \in P$ a maximum antichain of size $a$.  Let $c$ be the minimum number of disjoint chains required to fully partition $P$.  Since by definition, any two elements in a chain are comparable and conversely, any two elements in an antichain are incomparable, it follows that each chain of $P$ can contain at most one element from $A$.  Therefore $c \geq a$.\newline

ii)	Let $T(n)$ be the statement:  If $k$ is the size of the maximum antichain in a finite poset $P$, then $P$ can be subdivided into the disjoint union of $k$ chains.  We will proceed by induction on $n$, the number of elements in $P$.\newline

Base Case:  $T(n)$ is clearly true for both 0 element and 1 element posets.\newline
Induction hypothesis:  We assume that $T(n)$ is true for posets in which $\vert P \vert $ is a non-negative integer less than $n$.\newline
We now show that this implies $T(n)$ true for posets with $n$ elements.   Consider two cases:\newline

Case a):	Assume $P$ has a maximum antichain, $A$, of size $k$ in which at least one element is non-maximal, i.e. for some  $a_i \in  A, \exists y \in P$ such that $a_i<y$, and at least one element is non-minimal, i.e. for some $a_j \in  A, \exists z \in  P$ such that $a_j>z$.  Then we can split $P$ into two subsets $X$ and $Y$ where $X$ consists of elements that are $\geq a_i$ and $Y$ consists of elements that are $\leq a_j$.\newline
We claim:\newline i) $X\bigcap Y = A$\newline ii) both $X$ and $Y$ are non-empty since we assumed that $a_i, a_j$ exist \newline iii) both $X$ and $Y$ maintain the partial ordering in $P$ (each is itself a poset) \newline iv) $X$ and $Y$ have cardinality $< n$.\newline
Statements ii), iii) and iv) allow us to apply the inductive hypothesis to both subsets $X$ and $Y$.  This means that $X$ and $Y$ can each be decomposed into a union of $k$ chains.  But the minimal elements of the $k$ chains in $X$ and the maximal elements of the $k$ chains in $Y$ are all elements of $A$.  If we join a chain from $X$ and a chain from $Y$ via their common element in $A$, then $P$ can be spanned by $k$ chains.  Thus, $T(n)$ true for posets with cardinality up to $n$ implies that $T(n)$ is true for posets with cardinality equal to $n$. \newline

Case b):  Now assume that $P$ does not contain a maximum antichain meeting the above description.  Therefore any maximum antichain in $P$ must consist entirely of elements which are either maximal or minimal in $P$.  Let $x$ be a minimal element in $P$ and $y$ be a maximal element in $P$ such that $x \leq y$.  Then consider $P'$ to be $P\backslash\{x,y\}$.  Since we have deleted both a maximal and a minimal element from $P$, the largest antichain in $P'$ will contain $k-1$ elements.  Furthermore, by removing $x$ and $y$ we have created a poset in which $\vert P' \vert <n$, so by the induction hypothesis $P'$ can be decomposed into $k-1$ chains.  But if we add back the two element chain $x \leq y$ onto this total, we find that there are $k$ chains required to partition the original poset $P$.  Once again our assumption that $T(n)$ is true for posets with up to $n$ elements implies that $T(n)$ is true for posets with $n$ elements and therefore by the principle of mathematical induction, $T(n)$ is true for all non-negative integers, $n$.
\end{proof}
In contrast to this lengthy argument, the proof of the dual to this theorem is quite straightforward.  It is presented here since it gives us another way to link the data found for chains with that for antichains in $D_n$.
\text{ }
\\
\begin{theorem}
The size of any maximum chain in a finite poset $P$ is equal to the minimum number of antichains needed to fully partition $P$.
\end{theorem}
\begin{proof} (based on \cite{VLW})\newline
 Let $T(k)$ be the statement: If $P$ contains no chain of $k+1$ elements, then $P$ is the union of $k$ antichains.  Let $P$ be a finite poset with $M$ the set of maximal elements in $P$ and $k$ the length of the longest chain in $P$. Note that $M$ is an antichain. \newline
i) Base Case:  $T(k)$ is obviously true for $k=0$ and $k=1$. \newline
ii) Induction hypothesis:  Assume that $T(k)$ is true for non-negative integers up to $k$.\newline
iii) We show that this implies $T(k)$ is true for a poset $P$ that has no chain consisting of $k+1$ elements. Suppose that $x_1 < x_2 < \ldots < x_k$ is a chain in $P\backslash M$.  Then this would also be a maximal chain for $P$.  But this would imply that $x_k \in M$ which is a contradiction.  It follows that $P\backslash M$ does not contain a chain of $k$ elements.  Therefore by the induction hypothesis, $P\backslash M$ is the union of $k-1$ antichains.  Now by adding the antichain $M$ to this total, we see there are $k$ antichains in $P$.  So we have demonstrated, by the principle of mathematical induction, that $T(k)$ is true for all non-negative integers, $k$.
\end{proof}

\newpage
\section{Chromatic Polynomial}
We now consider the chromatic polynomials for the poset of Dyck paths ordered by inclusion, beginning with some standard graph theory definitions.
\begin{definition} A graph, $G$, is a collection of vertices and edges such that each edge joins exactly two vertices.
\end{definition}

\begin{definition} The chromatic polynomial, $C(G,t)$ is a polynomial in $t$ that represents the number of ways of colouring the vertices of the graph $G$, using $t$ colours, such that any vertices joined by an edge (i.e. adjacent vertices) have different colours.
\end{definition}

Since for any value of $n$ the poset of Dyck paths can be drawn as a graph with vertices representing the Dyck paths and edges illustrating the cover relationships, we are able to compute the corresponding chromatic polynomials.  This is most commonly done using an iterative technique \cite{BF},\cite{webCP} based on the following recurrence.
Let $C(G, t)$ be the chromatic polynomial for a graph $G$ containing an edge, $e$.  Then $$C(G,t)=C(G-e,t) - C(G\backslash e,t),$$ where:\newline
i) $G-e$ is the graph obtained from $G$ by deleting the edge, $e$, yet retaining its vertex endpoints.\newline
ii) $G\backslash e$ is the graph obtained from $G$ by contracting edge $e$ so that its vertex endpoints coincide.\newline
This recurrence must often be applied repeatedly in order to find the chromatic polynomial of a graph.
\begin{example}
Here, we illustrate how the recurrence can be used iteratively to find the chromatic polynomial for the poset of Dyck paths where $n=3$.\newline
Let $G$ be the poset of Dyck paths where $n=3$, as shown in the figure.\newline
\begin{figure}[htbp]
	\centering
		\includegraphics[width=1.2in]{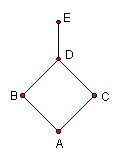}
	\caption{Dyck path poset for n=3}
\end{figure}

Then, by the recurrence, $C(G,t)=C(G-e,t) - C(G\backslash e,t)$.
Now, let $e$ be the edge joining vertices $B$ and $D$. Then, $G-e$ will be the graph that results when edge $BD$ is deleted from $G$ and $G\backslash e$ will be the graph obtained when $G$ is contracted along $BD$.\newline
\begin{figure}[htbp]
	\centering
		\includegraphics[width=2in]{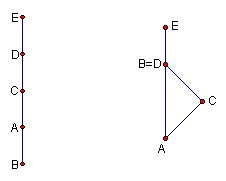}
	\caption{Graphs of G$-$e (left) and G$\backslash$e (right)}
\end{figure}\newline
\noindent For the graph of $G - e$, it can be easily shown that there are $t(t-1)^4$ possible colourings if $t$ colours are available.  There are $t$ choices to colour vertex $B$. For each subsequent vertex up the chain, our only restriction is that it not be coloured the same way as its predecessor.  Therefore, there are $t-1$ choices of colour for each of these vertices.
Thus, the equation is now:
$C(G,t)= t(t-1)^4 - C(G\backslash e,t)$.\newline
But, we can again apply the recursion on $C(G\backslash e,t)$. Let $f$ be the edge joining the vertex $D=B$ to the vertex $C$.  Then $(G\backslash e)-f$ is the resulting graph when edge $f$ is deleted from $G\backslash e$ and $(G\backslash e)\backslash f$ is the outcome of contracting $G\backslash e$ along $f$.  Both graphs are depicted below:\newline
\begin{figure}[htbp]
	\centering
		\includegraphics[width=110px]{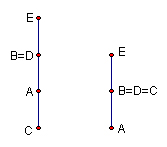}
	\caption{Graphs of (G$\backslash$e)$-$f    (left) and (G$\backslash$e)$\backslash$ f    (right)}
\end{figure}

\noindent By the same argument as above, $C((G\backslash e) - f, t) = t(t-1)^3$ and $C(G\backslash e)\backslash f, t)= t(t-1)^2$.  Putting it all together:
\begin{align*}
C(G,t)&= C(G-e,t) - C(G\backslash e,t)\\
&= t(t-1)^4 - C(G\backslash e,t)\\
&= t(t-1)^4 - [C((G\backslash e)- f)- C((G\backslash e)\backslash f)]\\
&= t(t-1)^4 - t(t-1)^3 + t(t-1)^2\\
&= t^5 - 5t^4 + 10t^3 - 9t^2 + 3t
\end{align*} 
\end{example}
While this recursive method will enable us to find the chromatic polynomial for the poset of Dyck paths for any $n$, it will be a time-consuming, and labour-intensive process.  Clearly this type of computation is ideally suited to a computer program such as Maple or SAGE.  Using the `with networks' environment in Maple, we are able to input the poset graphs and then have the computer output the chromatic polynomials.  The results for $n=0$ to $4$ are shown in the table below.\newline
\text{  }
\\
\begin{center}
\begin{tabular}[h]{l c}
\scriptsize
n &   \scriptsize $C(G, t)$\\
\hline
\scriptsize 0 & \qquad \scriptsize$t$\\
\scriptsize 1 & \qquad \scriptsize$t$\\
\scriptsize 2 & \qquad \scriptsize$t^2 - t$\\
\scriptsize 3 & \qquad \scriptsize$t^5 - 5t^4 + 10t^3 - 9t^2 + 3t$\\
\scriptsize 4 & \qquad \scriptsize$t^{14} - 21t^{13} + 210t^{12} - 1321t^{11} + 5823t^{10} - 18968t^9+46908t^8-89034t^7$\\
\scriptsize   & \qquad \scriptsize$+129490t^6-142270t^5 + 114532t^4 - 63791t^3 + 21940t^2 - 3499t$\\

\end{tabular}
\end{center}

\text{  }
\\
There does not appear to be any obvious pattern in these coefficients.  In addition, it is possible that other graphs could have the same chromatic polynomials as these Dyck path poset graphs.  The literature describes several cases where distinct graphs are shown to have the same chromatic polynomials \cite{BF}.  Whether this is true for the Dyck path poset chromatic polynomials remains an open question.  In the hopes that this issue will someday be resolved, we have added the sequence of coefficients from the above table to the Online Encyclopedia of Integer Sequences, as A141622 in \cite{OLEIS}.
\newpage
\section{Labelled Dyck paths and parking functions}

We now extend our investigation of the $D_n$ poset to look at labelled Dyck paths and their relationship to parking functions.

\begin{definition} Let $n$ be a fixed positive integer.  A parking function of order $n$ is a function  $f:(1,2,\ldots,n) \rightarrow (1,2,\ldots,n)$ such that $\{(x: f(x) \leq i)\} \geq i$ for $1\leq i\leq n$.\label{parkfun}  
\end{definition}

\noindent As in \cite{MB}, \cite{NL2}, \cite{NL1} we imagine the elements $x$ in the domain of $f$ to be cars and describe their standard parking procedure as follows.  Consider a one-way street with $n$ parking spaces, labelled $1,2, \dots, n$.  There are $n$ cars also labelled $1,2,\ldots,n$ which arrive in increasing numerical order to park in the spaces.  Each car $x$ has a favourite parking spot, $g(x)$, where it prefers to park.  When a car arrives, it first goes to its preferred spot, and parks there if the space is available. However, if the space is already occupied, the car continues forward along the street and parks in the next empty parking space, if one exists.  If a car reaches the end of the street without finding a parking space, then we say its parking attempt was unsuccessful.  But, if all $n$ cars have successfully found a parking spot at the end of this procedure, we say that $g$ is a parking function on $[n]$.

\begin{prop}
Let $P_n$ be the set of parking functions, $f$, of order $n$.  Then $\vert P_n\vert = (n+1)^{n-1}.$
\end{prop}
\begin{proof}
As in \cite{MB}, we prove this claim by considering a modified parking procedure.  Imagine that the street, while still one-way, is circular rather than linear and has $n+1$ parking spaces for the $n$ cars. There are $n+1$ possible preferred parking spots for each of the $n$ cars, which gives a total of $(n+1)^n$ choices. If a car is unable to successfully park in one of the first $n+1$ spaces, it leaves space $n+1$ and continues searching starting with space 1 again.  Clearly, all $n$ cars will have a parking space after this new procedure, and one parking space will remain empty.  Since cars do not leave their spots once they have parked, if it is space $n+1$ which is empty after all cars have parked, then all of the cars were able to successfully park in the first $n$ spots and so would also have been able to successfully park on the linear street with $n$ spaces. Therefore $f$ is a parking function on $[n]$ iff space $n+1$ remains empty at the end of this modified parking procedure.\newline
However, since each of the $n+1$ positions have an equal chance of being the one that remains empty, space $n+1$ will be left empty in $\frac{1}{n+1}$ of all cases.  Therefore the number of parking functions on $[n]$, $$  \vert P_n\vert = \frac{(n+1)^n}{n+1} = (n+1)^{n-1}.$$
\end{proof}
\noindent The integer sequence representing the number of parking functions on $n$ is A000272 in \cite{OLEIS}: 1, 1, 3, 16, 125, 1296, 16807, 262144, 4782969, ...  While the Catalan numbers (1, 1, 2, 5, 14, 42,\ldots) count ordinary (or unlabelled) Dyck paths, sequence A000272 will be shown to enumerate labelled Dyck paths.  In fact, it is noted in \cite{NL1} that there is a bijection between parking functions and labelled Dyck paths.  We begin by considering an example.
\begin{definition}
A labelled Dyck path is a Dyck path in which the $n$ steps in the north (or vertical) direction are labelled 1 to $n$ such that the labels of consecutive vertical steps increase from bottom to top in each column.
\end{definition}

\noindent These labels are placed in the partition cell that lies to the right of each vertical segment. 
\begin{example}  In the diagram below, a labelled Dyck path for $n = 6$ is shown.\newline
\label{labelledpath}
\begin{figure}[htbp]
	\centering
		\includegraphics[width=3in]{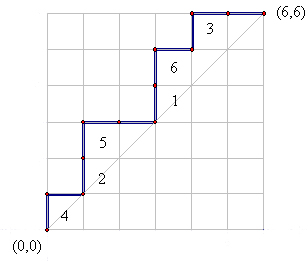}
	\caption{A labelled Dyck path for n = 6}
	\label{labelledDP}
\end{figure}
\end{example}
\noindent If we consider $P$ to be the labelled Dyck path, and number the columns from left to right in the diagram of $P$ with integers from 1 to $n$, then the labels in column $j$ correspond to the numbered cars that prefer to park in space $j$.  Therefore, the corresponding parking function, $f$, is obtained by setting $f(i)=j$ iff the label $i$ occurs in column $j$.  So, for the example above, the corresponding parking function is\newline
$$ f(1) = 4,\qquad  f(2) = 2,\qquad  f(3) = 5,\qquad
  f(4) = 1,\qquad  f(5) = 2,\qquad  f(6) = 4.$$
It is easy to verify that these satisfy definition \ref{parkfun} and therefore $f$ is a valid parking function.\newline
\text{  }
\\
Here is the procedure (from \cite{NL2}) for constructing a labelled Dyck path from a given parking function.  Let $f:(1,2,\ldots,n)\rightarrow(1,2,\ldots,n)$ be a parking function of order $n$.  Within this parking function, let $S_i=\{x:f(x)=i\}$ be the set of cars that prefer parking space $i$. We can construct the corresponding Dyck path using the following procedure. Label the columns in an $n \times n$ grid with the integers $1, 2, \ldots, n$.  Starting in the bottom row and with only one entry per row, place the numbers corresponding to $S_1$ in increasing order in column one of the grid.  Move to the next empty row and the second column of the grid and place the elements of $S_2$ in increasing order, one per row, in that column.  Continue this process in the remaining columns of the grid, always starting with the next empty row, until the elements of $S_n$ have been placed in increasing order in column $n$.  The path from $(0,0)$ to $(n,n)$ is created by drawing a vertical line along each cell boundary that is immediately to the left of the labels and then joining the vertical steps with horizontal ones as needed in order to make a connected path.  The lattice path created using this procedure is a Dyck path since:\newline
a)  A valid parking function will always have $f(x)=1$ for at least one value of $x$.  This ensures that the first step in the corresponding path will be in the vertical direction from the point $(0,0)$.  With each subsequent step, the path either travels vertically up (north) or horizontally to the right (east), and since for any parking function, $\{(x: f(x) \leq i)\} \geq i$ for $1\leq i\leq n$, the path will never pass below the line $y = x$.\newline
b)  A parking function of order $n$ contains $n$ values which will correspond to $n$ vertical steps in the diagram.  Thus, the path will contain $2n$ steps with $n$ in both the horizontal and vertical directions.  It will end at the point $(n,n)$ and will not travel outside the left or top boundaries of the grid.\newline

\begin{prop} There is a bijection between parking functions on $[n]$ and labelled Dyck paths of size $n$.
\end{prop}

\begin{proof} 
Let $f$ be a parking function of order $n$ and $D_n$ be a labelled Dyck path of size $n$.  We need to show that the function $g$ that maps the set of parking functions, $\{f_1, f_2, ..., f_k\}$, to the set of labelled Dyck paths in $D_n$ is both 1:1 and onto.\newline

\noindent(i)  $g$ is 1:1\newline
Let $f_1 = (a_1, a_2, \ldots, a_n)$ and $f_2 = (b_1, b_2, \ldots, b_n)$ be different parking functions of size $n$ where the $a_i$ and $b_i$ represent the preferred parking spaces for car $i$ under each function.  Suppose that these two parking functions mapped to the same labelled Dyck path, $D_n$.
Since $f_1$ and $f_2$ are different, then there must be at least one $a_i \in f_1 \neq b_i \in f_2$, i.e. at least one of the cars prefers a different parking space under parking function $f_1$ than it does under parking function $f_2$.  Suppose this difference occurs for the $j^{th}$ car.  Then, according to the procedure for constructing labelled Dyck paths from parking functions, the label $j$ will be placed in a different cell for $f_1$ than it will for $f_2$.  Since the same label is placed in different lattice cells under the two parking functions, they cannot correspond to the same labelled Dyck path.  Therefore we have contradicted the assumption that two different parking functions can be mapped to the same labelled Dyck path.  Thus, $g$ is 1:1.\newline

\noindent(ii)  $g$ is onto\newline
Let $D_n$ be a labelled Dyck path of size $n$.  Then it will consist of $n$ vertical (north) segments (beginning with one from (0,0) to (0,1)) and $n$ horizontal (east) segments such that the path never travels below the line $y = x$.  Each of the vertical segments will be labelled with an integer from 1, 2,\ldots, n.  Since there is a vertical segment along the left side of the bottom cell in column one, then $f(x) = 1$ for at least element, $x$, under the parking function $f$, i.e. at least one of the cars prefers to park in space one.  Since for any Dyck path the $y$-values must remain greater than or equal to the $x$ values, the path must contain a cumulative total of at least $i$ vertical segments before it horizontally crosses column $i$.  Thus, $\{x:f(x)\leq i\}\geq i$ for $1\leq i \leq n$, or equivalently, the number of cars preferring to park in spot 1 or 2 or $\ldots$ or $i \geq i$.  But, by definition, this means that $f$ is a valid parking function.  Therefore, $g$ is onto.
\end{proof}
\text{ }
\\
Recall that the area of a Dyck path (labelled or not) is the number of complete lattice squares that lie between the path and the line $y = x$.  Parking functions give us another method for finding the area of a Dyck path \cite{NL2}.  If we consider the set of all complete lattice cells which lie above the main diagonal in an $n \times n$ grid, it is straightforward to show that these can be counted by $$\sum_{i=1}^{n-1}i=1+2+3+...+(n-1)=\frac{n(n-1)}{2}.$$
But for any parking function, $f$, label $i$ will occur in column $f(i)$.  Therefore there are $f(i)-1$ lattice cells within the $n \times n$ grid and to the left of label $i$.  Since these cells are outside the Dyck path that corresponds to $f$, we subtract them from the total:
$$area(D)=\frac{n(n-1)}{2}-\sum_{i=1}^{n}(f(i)-1)=\frac{n(n+1)}{2}-\sum_{i=1}^{n}f(i).$$
Thus, in example \ref{labelledpath} (figure \ref{labelledDP}) above, we have:
$$area(D) = 21 - (4+2+5+1+2+4) = 3.$$
\text{  }
\\
\begin{definition} The area vector of a Dyck path is a sequence $g(D)=(g_1,\ldots,g_n)$, where $g_i$ is the number of complete lattice squares between the path and the line $y = x$ in row $i$ of the partition diagram where the rows are labelled 1 to $n$ from bottom to top.
\end{definition}
\text{  }
\\
\begin{definition} For a labelled Dyck path, the label vector $p(D)=(p_1,\ldots,p_n)$ is defined by letting $p_i$ be the unique label in row $i$ of the partition diagram.
\end{definition}
\text{  }
\\
\begin{example}
For the labelled Dyck path, $D$ in the figure below, area($D$)= 4.  The area vector $g(D)=(0,0,1,2,1)$ and the label vector $p(D)=(5,1,2,4,3)$.
\end{example}
\begin{figure}[htbp]
	\centering
		\includegraphics[width=2.2in]{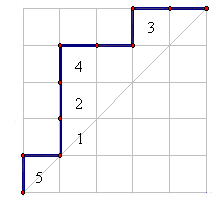}
	\caption{Labelled Dyck path with area = 4}
	\label{areavector}
\end{figure}
\noindent\normalfont
It is evident that any labelled Dyck path is completely determined by its area and label vectors \cite{NL1}.  Furthermore, in \cite{NL2} it is noted that the following conditions must hold for a pair of vectors $(g, p)$ corresponding to a labelled Dyck path of size $n$.\newline
\text {  }
\\
(1) $g$ and $p$ have length $n$.\newline
(2) $g_1 = 0$.\newline
(3) $g_i \geq 0$ for $1 \leq i < n$.\newline
(4) $g_{i+1} \leq g_i + 1$ for $1 \leq i < n$.\newline
(5) $p$ is a permutation of $\{1, 2, \ldots, n\}$.\newline
(6) $g_{i+1}=g_i + 1$ implies $p_i < p_{i+1}$.\newline
\text{  }
\\
As explained above, the set of all complete lattice cells which lie above the main diagonal in an $n \times n$ grid, is counted by $n(n-1)/2$.
This total can be split into two groups: the complete lattice cells which lie above the Dyck path and those that lie between the Dyck path and the main diagonal.  We have previously shown that the first group, which forms Young diagrams, is counted by the Catalan numbers for each $n$ since there is a bijection between the number of partitions above the Dyck path and the number of ordinary (or unlabelled) Dyck paths.  Now, the second group, whose elements form the area vectors, will be the complement of the Young partitions (above the main diagonal), and therefore will also be counted by the Catalan numbers.  So, by counting the number of distinct area vectors for any $n$, we will count the number of unlabelled Dyck paths.
\begin{example}
Consider the set of Dyck paths for $n = 3$.  By condition (1) above we know that any area vector, $g$, will have length 3.  By conditions (2), (3), and (4) we know that the only possible area vectors are these five: (0,0,0), (0,0,1), (0,1,0), (0,1,1) and (0,1,2). This corresponds with $C_3 = 5$. 
\end{example}
\noindent Alternatively, the number of unlabelled Dyck paths for any $n$ can be obtained from the set of labelled Dyck paths using a different type of label vector, called a column label vector.
\begin{definition}
A column label vector, $p'(D)=(p'_1,\ldots,p'_n)$ is defined by letting $p'_i$ be the column in which $i$ appears in the partition diagram.
\end{definition}
\noindent In figure \ref{areavector} above, the column label vector would be (2,2,4,2,1).  We note that other labellings of the same Dyck path would have column vectors composed of these same digits but arranged in a different order.  We will refer to the set of column label vectors with identical content but different ordering as a content group.  So, of the possible column vectors for labelled Dyck paths of size $n$, if we count only the ones from each content group in which the digits are arranged in their minimum order,  we will obtain the Catalan number sequence.  
\begin{example}
Here are the six labelled Dyck paths in $D_3$ with column label vectors in the \{1,2,3\} content group.  Of these, the red graph's column label vector has the minimum ordering, and therefore would be this group's representative.
\begin{figure}[htbp]
	\centering
		\includegraphics[width=3in]{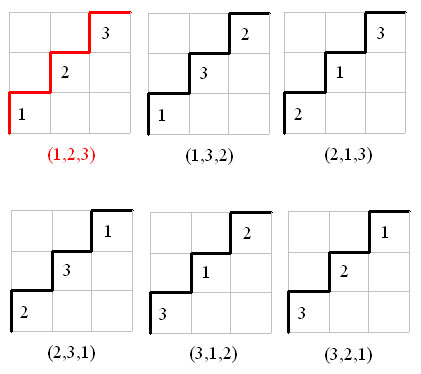}
	\caption{Labelled Dyck paths with column label vectors}
	\label{Dyck paths}
\end{figure}
\end{example}

However, this method of counting offers another advantage.  If we consider two column label vectors $p'(D_1)=(p'_1, p'_2,\ldots,p'_n)$ and $q'(D_2)=(q'_1, q'_2,\ldots,q'_n)$ where both $p'(D_1)$ and $q'(D_2)$ have the minimum ordering  for their content groups,  then $D_1 > D_2$ iff $p'_i \leq q'_i$ for all $i$ where $1 \leq i \leq n$.  We can recover the Dyck path poset (ordered by inclusion) by imposing this partial ordering on all of the column vectors which are least among their content groups for Dyck paths of size $n$.
\begin{example}
For the Dyck paths with $n = 3$, the column vectors which are least among their content groups are: (1,2,3),(1,1,3),(1,2,2),(1,1,2),(1,1,1).  Since each of these corresponds to a Dyck path, we know that there is a partial order given by $D_1 > D_2$ iff $p'_i \leq q'_i$ for all $i$ where $1 \leq i \leq n$ as illustrated in the diagram below.
\end{example}
\begin{figure}[htbp]
	\centering
		\includegraphics[width=1.5in]{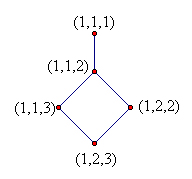}
	\caption{Dyck path poset n = 3}
	\label{Dyck path poset n = 3}
\end{figure}
\text{  }

\newpage
\section{Conclusion}
This paper has examined many of the basic properties of the poset of Dyck paths ordered by inclusion. We have considered a variety of methods for enumerating the poset and some important subsets within it such as chains and antichains.  In doing this, we have noted bijections between elements of the poset of Dyck paths and other combinatorial objects such as partitions, tableaux, other posets and graphs.  Finally, we extended the discussion to include labelled Dyck paths which led us to a study of parking functions.

However, several opportunities remain for future investigation.  Although we have added several entries to the Online Encyclopedia of Integer Sequences, \cite{OLEIS}, it would be desirable to extend some of these, in particular the antichain sequences A143673 and A143674.  The data for these two sequences were obtained by brute force counting methods using MuPAD:Combinat, \cite{webMuPAD}; it would be advantageous to find formulas that generate these terms.  In addition to those for antichains, maximum antichains and maximal antichains, we would also like to find explicit formulas for the total number of chains, and for chromatic and chain polynomials.  Although we have calculated these for values up to $n=5$, general formulas would allow us to determine higher values more expeditiously.  Another interesting challenge would be finding the bijection between the number of intervals in $D_n$ and the paths described in A005700 of \cite{OLEIS}.  Furthermore, since the poset of Dyck paths is just one of several combinatorial objects that are enumerated by the Catalan numbers, it would be interesting to consider the relationship between Dyck paths and the other objects.  For example, it would be nice to know if the $q$-analogs have any corresponding combinatorial interpretations in these other sets.  Such a discovery might lead to the resolution of open problems such as proving the $C_n(q,t)$ symmetry combinatorially.  Other avenues for future research include examining `nested' Dyck paths as has been started in \cite{MZ} and \cite{NLGW}, or enumerating the 3-D case where Dyck paths don't cross below the main diagonal in a cube.  Finally there is always the possibility that the work currently being done on Dyck paths will find practical application.  A few examples of this have already been seen, such as with chemical polymers building against a wall \cite{AR} or in models of light reflection and transmission through various media \cite{VS}.  Whatever the case may be, it is clear that there is still much to explore in the poset of Dyck paths ordered by inclusion.
 
\newpage
\section{Technical Acknowledgment}
Credit goes to Professor Mike Zabrocki for his assistance in writing MuPAD:Combinat programs for the antichains section of this paper.

\end{document}